%% file: revision.tex
\documentclass[a4paper,11pt,reqno]{amsart}

\usepackage[latin1]{inputenc} 
\usepackage{graphics}
\usepackage{times} 
\usepackage{amsfonts} 
\usepackage{multicol} 
\usepackage{amssymb} 
\usepackage{amsmath}
\usepackage{amssymb} 
\usepackage{bbm} 
\usepackage{enumerate}
\usepackage{natbib}
\usepackage{graphicx}

 \advance\topmargin -0.5cm
 \advance\oddsidemargin -1.0cm
 \advance\evensidemargin -1.0cm

\input{def}

\newtheorem{lem}{Lemma} 
\newtheorem{prop}{Proposition} 

\newtheorem{theo}{Theorem}
\theoremstyle{definition}
 
\newtheorem{assumption}{Assumption} 
\theoremstyle{remark}
\newtheorem{remark}{Remark}

\def\bPhi{\boldsymbol{\Phi}}
\def\bphi{\boldsymbol{\phi}}
\def\bbeta{\boldsymbol{\beta}}
\def\hbbeta{\boldsymbol{\widehat{\beta}}}
\def\hbbetadif{\boldsymbol{\widehat{u}}}
\def\cpout{\stackrel{P^\ast}{\longrightarrow}}
\def\param{\mathbf{t}}
\def\doubleparam{\mathbf{u}}
\def\pencontrast{\Lambda_n}
\def\pencontrastnorm{\mathbb{L}_n}
\def\modpencontrastnorm{\mathcal{L}_n}
\def\pencontrastlim{\mathbb{L}}
\def\modpencontrastlim{\mathcal{L}}
\def\paramset{\mathrm{T}}
\def\wc{\rightsquigarrow}
\def\Xset{\mathcal{X}}

\newcommand{\sgn}{\operatorname{{\mathrm sgn}}}

\title{Weak convergence of the regularization path in penalized M-estimation}

\date{June 10, 2009 (first revision)}

\author{Jean-Fran\c cois GERMAIN}
\address{RENAULT DREAM-DTAA, Technocentre Guyancourt, 1, avenue du Golf, 78288 Guyancourt, France.}
\email{jean-francois.germain@renault.com}

\author{Fran\c cois Roueff}
\address{Institut TELECOM, TELECOM ParisTech, LTCI CNRS, 46, rue Barrault, 75634 Paris Cedex 13, France}
\email{roueff@telecom-paristech.fr}



\begin{document}
\maketitle
\renewcommand{\thefootnote}{}
\footnote{\textit{Corresponding author}: F. Roueff,  Institut TELECOM, TELECOM ParisTech, LTCI CNRS.}
\renewcommand{\thefootnote}{\arabic{footnote}}

\begin{center}
{\it RENAULT DREAM-DTAA and Institut TELECOM, TELECOM ParisTech, LTCI CNRS}
\end{center}

\begin{abstract}
We consider an estimator $\hbbeta_n(\param)$ defined as the element 
$\bphi\in\bPhi$ minimizing a contrast process $\pencontrast(\bphi,\param)$ for each $\param$. 
We give some general results for deriving the weak convergence of $\sqrt{n}(\hbbeta_n-\bbeta)$ in the space of bounded functions,
where, for each $\param$, $\bbeta(\param)$ is the $\bphi\in\bPhi$ minimizing the limit of  $\pencontrast(\bphi,\param)$ as $n\to\infty$. 
These results are applied in the context of penalized M-estimation, that is, when
$\pencontrast(\bphi,\param)=M_n(\bphi)+\param J_n(\bphi)$, where $M_n$ is a usual contrast process and $J_n$ a penalty such
as the $\ell^1$ norm or the squared  $\ell^2$ norm.
The function $\hbbeta_n$ is then called a \emph{regularization path}.
For instance we show that the central limit theorem established 
for the lasso estimator in \cite{KNI00} continues to hold in a functional sense 
for the regularization path.
Other examples include various possible contrast processes for $M_n$ such as those considered in 
\cite{POL85}.
\end{abstract}

\section{Introduction}
\label{Introduction}

Let us consider a real-valued contrast process $\{M_n(\bphi),\;\bphi\in\bPhi\}$ based on an observed sample of size $n$ and  
a contrast function $M$ defined on the same parameter set $\bPhi$ and minimized at the point $\bbeta$.
A penalized estimator with penalty weight $\param\geq0$ is defined as the minimizer of the contrast process 
\begin{equation}
  \label{eq:DefPenContrast}
\pencontrast(\bphi,\param)=M_n(\bphi)+\param \; J_n(\bphi),\quad\bphi\in\bPhi\;,
\end{equation}
where $J_n$ is a non-negative function defined on $\bPhi$, not depending on the observations but possibly on $n$, mainly to
allow some appropriate normalization.   

The use of penalties is popular for ill-posed problems and model selection, among which the ridge
regression (see~\cite{hoerl70}) and the lasso  (see \cite{TIB96}) are emblematic examples.
In these two examples the contrast process $M_n$ is the least-square criterion and the penalty function $J_n$ is the
squared $\ell^2$ norm and the $\ell^1$ norm, respectively. 
Consistency and central limit theorems are established in \cite{KNI00} precisely in the case where $M_n$ is
the least-square criterion and $J_n$ is in a family of penalties including both the squared $\ell^2$ norm and the $\ell^1$
norm. They show that, when the penalty is properly normalized, the penalized mean square estimator 
is no longer asymptotically normal. Instead, its asymptotic distribution is given by the minimizer of
a penalized quadratic form depending on a Gaussian vector (see \textit{e.g.} \cite[Theorem~2]{KNI00}). Their asymptotic results
hold as the number $n$ of observations tends to infinity and for a fixed finite-dimensional model.  
Quite different results have been established when the dimension of the model increases with $n$, see \cite{GRE04,ZHA06,BUN07,BIC08} and the
references therein. These results provide interesting properties of the lasso for model selection or prediction purposes in the context of
sparse models. Although specific normalizations of the penalty (different from those required in~\cite{KNI00}) are prescribed
in these theoretical results, there exist numerous heuristic ways for choosing the penalty 
weight $\param$  in practice. The first step is to minimize $\pencontrast(\bphi,\param)$ in~(\ref{eq:DefPenContrast}) on
$\bphi\in\bPhi$ for a 
collection of non-negative weights $\param$, resulting in a collection of estimators $\hbbeta_n(\param)$, which is
called the \emph{regularization path} (or the \emph{solution path}). The Least Angle Regression (LAR) technique introduced by
\citeauthor{EFR04} in 
\cite{EFR04} provides, in most cases, the entire path,  computed with the complexity of a linear regression. 
In a second step, some criterion is used to select $\param$, see \textit{e.g.}
\cite{zou:hastie:tib:2007} where AIC and BIC procedures are proposed for the lasso. 
Because the whole path is used by the practitioner, we think that it is crucial to examine whether the convergence of
$\sqrt{n}(\hbbeta_n(\param)-\bbeta)$, established in \cite{KNI00} for one fixed $\param$, continues to hold in a 
functional sense and, if it is the case, to determine the limit distribution.
The goal of this paper is twofold. First we show that, under the same assumptions as in \cite{KNI00},
the convergence holds in the space of locally bounded functions. 
Second we extend this result to more general contrast processes $M_n$ such as generalized linear
models (GLM) or least amplitude deviation (LAD).  
A key result is a \emph{pathwise argmin theorem} which establishes the functional weak convergence of a path defined as the minimizer a
collection of contrast processes, see Theorem~\ref{theo:argmaxParam}.

For the moment let us give the asymptotic behavior of the lasso regularization path, which is the most simple application of
our results and which naturally extends~\cite{KNI00}. Consider the linear model 
\begin{equation}
		\label{eq:LinearModel}
y_k=\bx_k^T\bbeta+\varepsilon_k,\quad k=1,2,\dots
\end{equation}
where $\bbeta\in\Rset^p$ is an unknown parameter, $(y_k)$ is a sequence of real-valued observations, $(\bx_k)$ is
the sequence of regression vectors and $(\varepsilon_k)$ is a strong white noise  with variance $\sigma^2$.
For any $\param\geq 0$, the lasso estimator $\hbbeta_n(\param)$ 
minimizes the penalized contrast process $\Lambda_n(\bphi,\param)$ on $\bphi\in\Rset^p$, where 
\begin{equation}
		\label{eq:Contrastelasso}
		\Lambda_n(\bphi,\param)= \frac{1}{n}\sum_{k=1}^n(y_k-\bx_k^T\bphi)^2 + \param \lambda_n \sum_{i=1}^p|\phi_i| \; ,
\end{equation}
which is a specific form of (\ref{eq:DefPenContrast}). 
Denote $\bX_n=\left[\bx_1,...,\bx_n\right]^T$. We consider the following assumptions, for
consistency and central limit theorem, respectively. The assumptions are the same as in~\cite{KNI00}. 
\begin{assumption}
  \label{assump:Consistencelasso}\quad
\begin{enumerate}[(i)]
  \item\label{it:CondConsistencelasso1}
  $C_n=n^{-1} \bX_n^T\bX_n\to C$, where $C$ is a positive-definite matrix;
  \item\label{it:CondConsistencelasso2}
  $\lambda_n\to 0$.
\end{enumerate}
\end{assumption}
\begin{assumption}
  \label{assump:CLTlasso}\quad
\begin{enumerate}[(i)]
  \item\label{it:CondCLTlasso1} Assumption~\ref{assump:Consistencelasso}-(\ref{it:CondConsistencelasso1}) holds;
  \item\label{it:CondCLTlasso2} $\max_{1\leq k\leq n} \|\bx_k\|^2=o(n)$;
  \item\label{it:CondCLTlasso3} $\lambda_n=n^{-1/2}$.
\end{enumerate}
\end{assumption}
% Using the same assumptions on $\boldsymbol{X}_n$ and $\lambda_n$ as in \cite{KNI00},
% we show that the consistency and the CLT hold for the lasso regularization path. 
Assumptions \ref{assump:Consistencelasso}-(\ref{it:CondConsistencelasso1}) and 
\ref{assump:CLTlasso}-(\ref{it:CondCLTlasso2}) are the classical assumptions for the asymptotic behavior of least squares
estimators. The other assumptions provide the appropriate way of normalizing the $\ell^1$ penalty. 
\begin{theo}
\label{theo:lassoConsistence}
Under Assumption~\ref{assump:Consistencelasso}, $\hbbeta_n(\param)$ converges in probability to $\bbeta$ locally uniformly in
$\param\in\Rset_+$, that is 
\begin{equation}
  \label{eq:consistencelasso}
\hbbeta_n\cp \bbeta  \; \textrm{ in } \ell^\infty_o(\Rset_+,\Rset^p) \; ,  
\end{equation}
where $\ell^\infty_o(\Rset_+,\Rset^p)$ denotes the space of locally bounded $\Rset_+\to\Rset^p$ functions. 
\end{theo}
We now define the limit process of the lasso regularization path, appropriately centered and normalized.
Let $U\sim\mathcal{N}(0,\sigma^2 C)$. For any $\param\geq 0$, we define $\hbbetadif(\param)$ as the
point $\bphi\in\Rset^p$ which minimizes 
\begin{equation}
	\label{eq:ContrasteLimite}
	\pencontrastlim(\bphi,\param)=-2U^T\bphi + \bphi^T C \bphi + \param 
	\left[\sum_{j=1}^p \phi_j \sgn\left(\beta_j\right) \1_{\{\beta_j \neq 0\}} + |\phi_j| \1_{\{\beta_j = 0\}} \right] \; .
\end{equation}
It is easy to show that this defines $\hbbetadif(\param)$ uniquely for all $\param\geq 0$ (see the proof of
Theorem~\ref{theo:lassoCLT}). The distribution of $\hbbetadif$ as a function is not explicit but is not more complicated than
its marginal distributions already described in~\cite{KNI00}, since the whole path is described as a deterministic function
of the random variable (r.v.) $U$. An interesting property  of $\hbbetadif(\param)$ is that, with probability 1, 
the set of its components
that vanish for $\param$ large enough is given by the set of zero components of the true parameter $\bbeta$. 

\begin{theo}
\label{theo:lassoCLT}
Under Assumption~\ref{assump:CLTlasso},
\begin{equation}
  \label{eq:lassoCLTStochDiff}
\sqrt{n}(\hbbeta_n-\bbeta)\wc\hbbetadif  \; \textrm{ in } \ell^\infty_o(\Rset_+,\Rset^p) \; ,
\end{equation}
where $\wc$ denotes the weak convergence.
\end{theo}
 
\begin{remark}
The convergence in $\ell^\infty_o(\Rset_+,\Rset^p)$ is equivalent to the uniform convergence on every compact subset of
$\Rset_+$. In fact the convergences~(\ref{eq:consistencelasso}) and~(\ref{eq:lassoCLTStochDiff}) cannot be improved in the
sense that they do not hold uniformly on $\Rset_+$. To see why, observe that, by the definition of
$\hbbetadif$, its coordinates corresponding to non-vanishing $\beta_j$ are unbounded as $\param\to\infty$.  
In contrast, the left-hand side of~(\ref{eq:lassoCLTStochDiff}) is bounded since, for any $n$, there is  
a large enough $\param$ for which $\hbbeta_n(\param)=0$. Note that this also  implies that 
$\sup_{\param\in\Rset_+}\|\hbbeta_n(\param)-\bbeta\|\geq\|\bbeta\|$, and thus that the
consistence~(\ref{eq:consistencelasso}) does not hold if $\ell^\infty_o(\Rset_+,\Rset^p)$ is replaced by the set of
bounded $\Rset_+\to\Rset^p$ functions $\ell^\infty(\Rset_+,\Rset^p)$ endowed with the sup norm. 
\end{remark}

The proofs of Theorem~\ref{theo:lassoConsistence} and Theorem~\ref{theo:lassoCLT} are applications of some general
results on the consistency of convex penalized M-estimators and on the weak convergence of Argmin's 
depending on an tuning parameter $\param$ (the so called pathwise argmin theorem in the following). More general penalized
contrasts will also be considered. Such extensions are of interest since the lasso regularization path has been
extended  to the case where $M_n$ is different from the least-square  criterion.
In \cite{PAR06}, a fast numerical algorithm is proposed for determining the regularization path when $M_n$ is
a regression function based on a negated log-likelihood of the  canonical exponential family.  
In \cite{GER07}, a fast algorithm based on a dichotomy  is proposed
to explore the range of $\param$'s in the specific case of logistic regression penalized by the $\ell^1$ norm.

The paper is organized as follows. In Section~\ref{sec:centr-limit-theor}, we provide a pathwise argmin theorem 
(Theorem~\ref{theo:argmaxParam}). Section~\ref{sec:penal-m-estim} is concerned with the asymptotic behavior of the
regularization  path of a penalized contrast. Very mild conditions on the contrast and on the penalty 
are provided for obtaining the uniform consistency and the central limit of the path and a particular attention is given to
the case where both the contrast and the penalty are convex. Except for the convex case, such results can actually be seen as
special cases of the more general study of pathwise M-estimators, which is treated in   
Section~\ref{sec:MestimCLT}. Finally we provide several examples of applications of these results in
Section~\ref{sec:OtherExamples}, including the $\ell^1$-penalized general linear model (GLM) introduced in~\cite{PAR06},
the penalized least absolute deviation (LAD) and the Akaike information criterion. 
Concluding remarks are provided in Section~\ref{sec:conclusion}. The detailed proofs are deferred to the appendix for
convenience.

\section{A pathwise argmin theorem}
\label{sec:centr-limit-theor}

To obtain a CLT for the regularization path, we rely on a  pathwise argmin theorem, which is of independent interest,
and can be seen as an extension of \cite[Theorem~2.7]{kim:pollard:1990} (see also
\cite[Theorem~3.2.2]{vandervaart:wellner:1996}) to fit the context of a path defined as the minimizer of a
collection of contrast processes. 

Let us recall some of the terminology and notation used in~\cite{vandervaart:wellner:1996}. 
For a metric space $\calD$, we say that a sequence of $\calD$-valued maps $(X_n)$ defined on $\Omega$ converges weakly to a
$\calD$-valued map $X$ defined on $(\Omega,\calF)$, and denote $X_n\wc X$, if $X$ is a Borel map and, for any real-valued
bounded continuous function $f$ defined on $\calD$,
$$
E^*[f(X_n)] \to E[f(X)]\;,
$$
where $E$ denotes the expectation with respect to $P$ and $E^*$ denotes the \emph{outer expectation}, defined for every real-valued
map $Z$ defined on $\Omega$ by $E^*[Z]=\inf\{E[U]~:~U\geq Z\}$, where, in this sup, the r.v. $U$ is taken measurable. The
\emph{inner expectation} and \emph{inner probability} are 
respectively defined by $E_*[Z]=-E^*[-Z]$ and $P_*(A)=1-P^*(A^c)$, where
$A^c$ denotes the complementary set of $A$ in $\Omega$.

For any positive integer $p$ and any set $\paramset$ we further denote by 
$\ell^\infty(\paramset,\Rset^p)$ the normed space of bounded functions $f=(f_1,\dots,f_p)$ taking values in $\Rset^p$ and defined on $\paramset$ endowed with
the sup norm on $\paramset$, denoted by 
$$
\|f\|_\paramset=\sup_{t\in \paramset, i\in\{1,\dots,p\}}|f_i(t)| \; .
$$
We will simply denote  $\ell^\infty(\paramset,\Rset^p)$ by $\ell^\infty(\paramset)$ for $p=1$.
\begin{theo}\label{theo:argmaxParam}
Let $\bPhi$ be a metric space endowed with a metric $d$ and $\paramset$ be an arbitrary set.   
We suppose that we are in one of the two following cases  
\begin{enumerate}[(C-1)]
\item\label{it:finite-dim-case}  $\paramset$ is a finite set. In this case, we set $\calD=\bPhi^\paramset$ endowed with the product
  topology; 
\item\label{it:LinftyConvinRp} $\bPhi=\Rset^p$ with $p\geq1$, $d$ being the Euclidean metric. In this case, we set
  $\calD=\ell^\infty(\paramset,\Rset^p)$. 
\end{enumerate}
Let $\{\pencontrastnorm(\bphi,\param),\;\bphi\in\bPhi,\param\in\paramset\}$ be a sequence of real-valued processes,
$\{\pencontrastlim(\bphi,\param),\;\bphi\in\bPhi,\param\in\paramset\}$ be a real-valued process,
 $\{\hbbetadif(\param),\;\param\in\paramset\}$ be a $\bPhi$-valued process,
and $\{\hbbetadif_n(\param),\;\param\in\paramset\}$ be a sequence of $\bPhi$-valued processes.
Assume that 
\begin{enumerate}[(i)]
\item\label{it:LinftyConvContrast} for any compact set $K\subset\bPhi$, $\pencontrastnorm\wc\pencontrastlim$ in
  $\ell^\infty(K\times\paramset)$ and $\pencontrastlim$ is a tight Borel map taking values in  $\ell^\infty(K\times\paramset)$;
\item\label{it:isolatedMinimumlimitCLT} for any $\eta>0$ and compact $K\subset\bPhi$, we have almost surely that
  \begin{equation}
    \label{eq:isolatedMinimumlimitCLT}
\inf_{\param\in\paramset}\left[\inf\{\pencontrastlim(\bphi,\param)~:~\bphi\in K,\;d(\bphi,\hbbetadif(\param))\geq\eta\}-\pencontrastlim(\hbbetadif(\param),\param)\right]
>0 \;;     
  \end{equation}
\item\label{it:limitTight} %$\hbbetadif$ is a tight Borel map in $\calD$.
for any $\epsilon>0$, there exists a compact $K\subset\bPhi$ such that
\begin{equation}
  \label{eq:Limtighfidi}
P\left(\hbbetadif(\param)\in K\text{ for all }\param \in\paramset\right) \geq 1-\epsilon \; ;
\end{equation}
\item\label{it:UnifTight} for any $\epsilon>0$, there exists a compact $K\subset\bPhi$ such that
\begin{equation}
  \label{eq:Uniftighfidi}
\liminf P_*\left(\hbbetadif_n(\param)\in K\text{ for all }\param \in\paramset\right) \geq 1-\epsilon \; ;
\end{equation}
\item   \label{it:DefArgminFalCLT} $\hbbetadif_n$ is approximately minimizing $\pencontrastnorm$,
\begin{equation}
  \label{eq:DefArgminFalCLT}
\sup_{\param\in\paramset}\left\{\pencontrastnorm(\hbbetadif_n(\param),\param)
-\inf_{\bphi\in\bPhi}\pencontrastnorm(\bphi,\param)\right\}_+ = o_{P^*}(1) \; .  
\end{equation}
\end{enumerate}
Then there is a version of $\hbbetadif$ in $\calD$ and $\hbbetadif_n\wc\hbbetadif$.
\end{theo} 
% \begin{remark}
% Condition~(\ref{it:isolatedMinimumlimitCLT}) is satisfied if, almost surely, $\hbbetadif(\param)$ is the unique maximum of 
% $\pencontrastlim(\cdot,\param)$ for all $\param\in\paramset$ and $\{\pencontrastlim(\cdot,\param),\,\param\in\paramset\}$ 
% is equi-lower semicontinuous. The latter means that for every $\epsilon>0$ and $\bphi_0\in\bPhi$, there exists $\eta>0$ such that 
% $\pencontrastlim(\bphi,\param)\geq \pencontrastlim(\bphi_0,\param)-\epsilon$ for all $\bphi$ such that
% $d(\bphi,\bphi_0)\leq\eta$ and all $\param\in\paramset$.
%  Theorem~3.2.2 in~\cite{vandervaart:wellner:1996},
% \end{remark}
\begin{proof}[Proof of Theorem~\ref{theo:argmaxParam} in the case~(C-\ref{it:finite-dim-case})]
In the case~(C-\ref{it:finite-dim-case}), where $\paramset$ is finite, for any compact $K\subset\bPhi$, 
$K^\paramset$ is a compact subset of $\bPhi^\paramset$ endowed with the metric
$$
d_\paramset(\mathbf{u},\mathbf{v})=\sup_{\param\in \paramset}d(\mathbf{u}(\param),\mathbf{v}(\param))\;.
$$
Hence, in this case, Conditions~(\ref{it:limitTight}) and~(\ref{it:UnifTight}) respectively say that $\hbbetadif$ is tight
and $(\hbbetadif_n)$ is uniformly tight in $\bPhi^\paramset$.  
In this case, the conclusion of Theorem~\ref{theo:argmaxParam} follows almost directly from Theorem~3.2.2
in~\cite{vandervaart:wellner:1996}. To see why, let us introduce the following contrast process 
\begin{equation}
  \label{eq:modif_contrast}
  \modpencontrastnorm(\mathbf{v}) = \sup_{\param\in\paramset}\left\{\pencontrastnorm(\mathbf{v}(\param),\param)
-\inf_{\bphi\in\bPhi}\pencontrastnorm(\bphi,\param)\right\},\quad\mathbf{v}\in\bPhi^\paramset \; .
\end{equation}
Observe that defining $\hbbetadif_n(\param)$ as a minimizer of $\pencontrastnorm(\cdot,\param)$ for all $\param\in\paramset$
is equivalent to defining $\hbbetadif_n$ directly as a minimizer of $\modpencontrastnorm$.
In particular Condition~(\ref{it:DefArgminFalCLT}) implies that 
$$
\modpencontrastnorm(\hbbetadif_n)\leq \inf_{\mathbf{v}\in\bPhi^\paramset}\modpencontrastnorm(\mathbf{v}) + o_{P^*}(1) \;,
$$
that is, $\hbbetadif_n$ is a near minimizer of $\modpencontrastnorm$.
Condition~(\ref{it:LinftyConvContrast}), in turn, by the continuous mapping theorem, 
implies that $\modpencontrastnorm \wc \modpencontrastlim$ in $\ell^\infty(\bPhi^\paramset)$, where, for any
$\mathbf{v}\in\bPhi^\paramset$, 
$$
\modpencontrastlim(\mathbf{v})=\sup_{\param\in\paramset}\left\{\pencontrastnorm(\mathbf{v}(\param),\param)
-\inf_{\bphi\in\bPhi}\pencontrastnorm(\bphi,\param)\right\}\;.
$$
Finally it is not too difficult to show that Condition~(\ref{it:isolatedMinimumlimitCLT}) implies 
that, almost surely, for all compact $K\subset\bPhi$ and $\eta>0$,
$$
\inf\left\{\modpencontrastlim(\mathbf{v})~:~\mathbf{v}\in K^\paramset,\,d_\paramset(\mathbf{v},\hbbetadif)\geq\eta\right\} >
0=\modpencontrastlim(\hbbetadif)\;.
$$
This condition corresponds to the semicontinuity and argmax uniqueness conditions appearing in Theorem~3.2.2
in~\cite{vandervaart:wellner:1996}. Hence this theorem applies and yields $\hbbetadif_n\wc\hbbetadif$ in the
case~(C-\ref{it:finite-dim-case}). 
\end{proof}

The proof in the case~(C-\ref{it:LinftyConvinRp}) is postponed to the appendix. The main originality of
Theorem~\ref{theo:argmaxParam} lies in the 
case~(C-\ref{it:LinftyConvinRp}). In this case, Theorem~3.2.2
in~\cite{vandervaart:wellner:1996} cannot be directly applied because Condition~(\ref{it:UnifTight}) is no longer a uniform tightness condition
($K^\paramset$ is not a compact subset of $\ell^\infty(\paramset,\Rset^p)$). The key idea, detailed
in the appendix, is to show that, under Conditions~(\ref{it:isolatedMinimumlimitCLT})--(\ref{it:DefArgminFalCLT}), this
asymptotic tightness of $\hbbetadif_n$ in  $\ell^\infty(\paramset,\Rset^p)$ is inherited from that of $\pencontrastnorm$
assumed in Condition~(\ref{it:LinftyConvContrast}).

\section{Penalized  M-estimation}
\label{sec:penal-m-estim}
\subsection{Uniform consistency}
\label{Consistency}

Standard results on the consistency of M-estimators (see \textit{e.g.}~\cite[Theorem~5.7]{vandervaart:1998}) 
roughly say that if $\hbbeta_n$ is a sequence of minimizers of $M_n$ on $\bPhi$, $M_n$ tends to $M$ with some uniformity and $\bbeta$
is an isolated minimum of $M$ on $\bPhi$, then $\hbbeta_n$ converges to $\bbeta$ in probability. 
We will use the following set of conditions which are slightly weaker than the classical ones. 
\begin{assumption}\label{assum:uniform-consistency}
There exists $\bbeta\in\bPhi$ such that 
  \begin{enumerate}[(i)]
  \item\label{it:consUnfiLowBound} $\displaystyle\sup_{\bphi\in\bPhi}\left\{M(\bphi)-M_n(\bphi)\right\}_+\cp 0$, where
    $a_+=\max(0,a)$ for any $a\in\Rset$; 
  \item\label{it:SimpleConv} $M_n(\bbeta)\cp M(\bbeta)$;
  \item \label{it:IsolMin} for all $\epsilon>0$, $\displaystyle\inf\{M(\bphi)~:~\bphi\in\bPhi,\;d(\bphi,\bbeta)\geq\epsilon\}>M(\bbeta)$,
  \end{enumerate}
where $d$ is a metric endowing the metric space $\bPhi$. 
\end{assumption} 

Let us briefly comment these assumptions.
Conditions~(\ref{it:consUnfiLowBound}) and~(\ref{it:SimpleConv}) are generally replaced by the stronger uniform convergence
condition $\sup_{\bphi\in\bPhi}\left|M(\bphi)-M_n(\bphi)\right|\cp 0$. These weaker conditions are for instance useful when $\bPhi$ is
non-compact since it is then sufficient to show the uniform convergence on a compact subset and provide a lower bound of $M_n$
out of this compact. 
%The point-wise convergence in~(\ref{it:SimpleConv}) is usually obtained using the weak law of large numbers. 
Condition~(\ref{it:IsolMin}) is the standard condition which defines $\bbeta$ as the (unique) isolated minimum of the limit
contrast function.  

We will show that, under Assumption~\ref{assum:uniform-consistency}, provided that $J_n(\beta)$ tends to 0, the minimizer
$\hbbeta_n(\param)$ of $\pencontrast(\bphi,\param)$ converges to $\bbeta(\param)$, \emph{locally uniformly} in $\param$. 
To avoid making measurability assumptions on the path $\param\mapsto\hbbeta_n(\param)$, we need to work with outer
probability to extend the probability to possibly non-measurable sets. Given a probability space $(\Omega,\calF,P)$, we
denote by $P^\ast$ the \emph{outer probability} defined on the subsets of $\Omega$ by
$$
P^\ast(A)=\inf\{P(B):\;B\in\mathcal{F}\text{ with }A\subset B\},\quad A\subseteq\Omega \;. 
$$
We say that a sequence $(Y_n)$ of real-valued maps defined on $\Omega$  converges in $P^\ast$--probability to
0 and denote $Y_n\cpout0$ if, for any $\epsilon>0$, $P^\ast(\{|Y_n|\geq\epsilon\})\to0$. Here $\{|Y_n|\geq\epsilon\}$ is the
usual short-hand notation for the subset $\{\omega\in\Omega~:~|Y_n(\omega)|\geq\epsilon\}$.
When $Y_n$ is measurable as a
map taking values in $\Rset$ endowed with the Borel $\sigma$-field, this
is equivalent to the usual convergence in probability. 

\begin{theo} \label{theo:uniform-consistency}
Suppose that Assumption~\ref{assum:uniform-consistency} holds for some $\bbeta\in\bPhi$, $M$ defined on $\bPhi$ and
$\{M_n(\bphi),\;\bphi\in\bPhi\}$, a sequence of real-valued processes.  
Let $(J_n)$ be a sequence of non-negative functions defined on $\bPhi$ such that $J_n(\bbeta)\to0$.  
Let $\paramset$ be a compact subset of $[0,\infty)$ and suppose that we have a $\bPhi$-valued process $\{\hbbeta_n(\param),\;\param\geq0\}$ such that
\begin{equation}
  \label{eq:DefMestimaPen}
\sup_{\param\in\paramset}  \left\{\pencontrast(\hbbeta_n(\param),\param)-\pencontrast(\bbeta,\param)\right\}_+\cpout 0\;,
\end{equation}
where $\pencontrast$ is defined by~(\ref{eq:DefPenContrast}).
Then $\hbbeta_n(\param)$  converges to $\bbeta$ uniformly in $\param\in\paramset$, in $P^\ast$--probability, that is, 
\begin{equation}
  \label{eq:ConvUnifMestPen}
  \sup_{\param\in\paramset}d(\hbbeta_n(\param),\bbeta)\cpout 0\;.
\end{equation}
\end{theo}
\begin{remark}
In statistical applications the contrast function $M$ in
Assumption~\ref{assum:uniform-consistency} depends on the unknown distribution of the contrast process $M_n$ and thus  
$\bbeta$ is an unknown point of $\bPhi$. In particular, the convergence condition $J_n(\bbeta)\to0$ has to be verified for any 
$\bbeta\in\bPhi$ (but not uniformly in $\bbeta$) and it simply amounts to correctly 
normalize the penalty $J_n$ as $n\to\infty$.
\end{remark}

\begin{remark}
The same result holds if the convergence in $P$-probability in
Assumption~\ref{assum:uniform-consistency}-(\ref{it:consUnfiLowBound}) is replaced by a convergence in $P^\ast$-probability. 
However, in applications, the smoothness properties of $\bphi\mapsto M_n(\bPhi)$ and $\bphi\mapsto M(\bPhi)$ usually imply
that $\sup_{{\bphi}\in{\bPhi}}\{M(\bphi)-M_n(\bphi)\}_+$  is a measurable function.  
\end{remark}
\begin{remark}
The fact that the outer probability $P^*$ appears in~(\ref{eq:DefMestimaPen}) does not bring real
difficulties in applications.
Indeed Condition~(\ref{eq:DefMestimaPen}) follows from the definition of $\hbbeta_n(\param)$ as a near minimizer
of $\pencontrast(\cdot,\param)$, that is, if  $\hbbeta_n(\param)$ satisfies
$$
\pencontrast(\hbbeta_n(\param),\param)\leq \inf_{\bphi\in\bPhi} \pencontrast(\bphi;\param) + u_n \;,   
$$
with $u_n=o_P(1)$ not depending on $\param$, \textit{e.g.} $u_n=0$ (perfect minimizer) or $u_n=n^{-1}$ (near
minimizer). The numerical computation of a near minimizer is a difficult task in general, in particular in the presence of  
several local minima. We will focus on convexity assumptions in Section~\ref{sec:ConvexCase}, which cover many cases of
interest and which usually allow tractable numerical computation of $\hbbeta_n(\param)$ for any $\param$. 
\end{remark}
\begin{remark}
Although $\hbbeta_n(\param)$ is an r.v.  for any $\param$,  the map
$\sup_{\param\in\paramset}\|\hbbeta_n(\param)-\bbeta\|$ defined on $\Omega$ may not be measurable (it is in some
particular cases, for instance if the map $\param\mapsto\hbbeta_n(\param)$ is continuous). This is where
the outer probability is useful. Nevertheless, for any $\param\geq0$, the event
$\{d(\hbbeta_n(\param),\bbeta)\geq\epsilon\}$ is measurable, and its probability is less than the left-hand side of
Eq.~(\ref{eq:ConvUnifMestPen}); hence, for any $\param\geq0$, 
$\hbbeta_n(\param)\cp\bbeta(\param)$. 
%This is the consistency result obtained by \citeauthor{KNI00} in \cite{KNI00} in the lasso case.
\end{remark}
\begin{remark}
For $L=0$ in~(\ref{eq:ConvUnifMestPen}), we get a standard result on the consistency of M-estimators (without penalty). It is important to
notice that the consistency of penalized M-estimators is obtained for free, in the sense that no additional assumption
on $M_n$ or $M$ is required and the only assumption on $J_n$ is  $J_n(\bbeta)\to0$.  
\end{remark}

\subsection{Uniform consistency in the convex case}
\label{sec:ConvexCase}

In this section, we consider the following assumption.  
\begin{assumption}[convexity assumption]
  \label{assump:convex}
$\bPhi$ is a convex subset of an Euclidean space endowed with the norm $\|\cdot\|$ and $M_n$ is a convex real-valued function
on $\bPhi$ almost surely. Let $V\subseteq\bPhi$ be a neighborhood of the point $\bbeta$ and $\Delta$
be a  strictly convex real-valued function  defined on $V$ such that   
\begin{enumerate}[(i)]
%\item\label{it:convex-g} for any $(x,y)\in\calX\times\calY$, $\delta((x,y),\cdot)$ is a convex function on $\bPhi$;
\item\label{it:pointwise-conv} for any $\bphi\in V$, $M_n(\bphi) \cp \Delta(\bphi)$;
\item\label{it:convex-minimum} $\Delta(\bphi)\geq\Delta(\bbeta)$ for all $\bphi\in V$.
\end{enumerate}
\end{assumption}

Convex M-estimation is considered in \cite{haberman:1989} and somewhat simplified in \cite{niemiro:1992}.
In the following result the convexity assumption is twofold. First it implies
Assumption~\ref{assum:uniform-consistency}. Second, if the contrast the penalty $J_n$ is
strictly convex, then the minimization of~(\ref{eq:DefPenContrast}) has a unique solution and
this solution path is continuous, which allows to replace the outer probability in~(\ref{eq:ConvUnifMestPen}) by a
standard probability. 
Convexity is also useful in practice since $\hbbeta_n(\param)$  can  be computed using efficient numerical procedure for
convex optimization (see \cite{boyd:vandenberghe:2004}). 

\begin{theo}\label{theo:ConvexConsistence}
Suppose that Assumption~\ref{assump:convex} holds.
Let $(J_n)$ be a sequence of non-negative functions defined on $\bPhi$ such that $J_n(\bbeta)\to0$ and define $\pencontrast$ as
in~(\ref{eq:DefPenContrast}).   
Then the 3 following assertions hold. 
\begin{enumerate}[(a)]
\item\label{it:ConvUnifMestPenConvex} For any $L\geq0$, if we have a $\bPhi$-valued process $\{\hbbeta_n(\param),\;\param\geq0\}$
  satisfying~(\ref{eq:DefMestimaPen}), $\hbbeta_n(\param)$  converges to $\bbeta$ uniformly in $\param\in[0,L]$, in  
$P^\ast$--probability, that is,~(\ref{eq:ConvUnifMestPen}) holds.
\item\label{it:StrictlyConvex} If $J_n$ is strictly convex on $\bPhi$, then it is always possible to define a deterministic
  non-negative sequence $(L_n)$ with $L_n\to\infty$, a sequence $(A_n)$ of events in $\calF$ with $P(A_n)\to1$, and, for each 
  $n$, a collection $\{\hbbeta_n(\param),\;\param\geq0\}$ of r.v.'s satisfying the two
  following properties.
\begin{enumerate}[(\ref{it:StrictlyConvex}1)]
\item\label{it:UniqueMinProperty} For all $\param\in[0,L_n]$ and $\omega\in A_n$,  
  $\pencontrast(\hbbeta_n(\omega,\param),\param)$ is a minimum of $\pencontrast(\bphi,\param)$ on $\bphi\in\bPhi$ and this
  minimum is unique for $\param>0$.  
\item\label{it:ConvexContProperty} For all $\omega\in \Omega$, $\hbbeta_n(\omega,\cdot)$ is a continuous function on
  $(0,L_n]$ and on $(L_n,\infty)$.
\end{enumerate}
As consequences,~(\ref{eq:DefMestimaPen}) holds for any $L>0$ and the uniform convergence~(\ref{eq:ConvUnifMestPen}) 
  holds in $P$--probability, that is, 
\begin{equation} 
  \label{eq:ConvUnifMestPenStandardProb}
  \sup_{\param\in[0,L]}\|\hbbeta_n(\param)-\bbeta\|\cp 0\;.
\end{equation}
\item\label{it:MStrictlyConvex}  If $M_n$ is strictly convex on $\bPhi$ for all $n$, then the conclusions of (\ref{it:StrictlyConvex})
  hold with Properties~(\ref{it:UniqueMinProperty}) and~(\ref{it:ConvexContProperty}) strengthened as follows.
\begin{enumerate}[(\ref{it:MStrictlyConvex}1)]
\item\label{it:UniqueMinPropertyPrim} For all $\param\in[0,L_n]$ and $\omega\in A_n$,  
  $\pencontrast(\hbbeta_n(\omega,\param),\param)$ is the unique minimum of $\pencontrast(\bphi,\param)$ on $\bphi\in\bPhi$.
\item\label{it:ConvexContPropertyPrim} For all $\omega\in \Omega$, $\hbbeta_n(\omega,\cdot)$ is a continuous function on
  $[0,L_n]$ and on $(L_n,\infty)$.
\end{enumerate}
\end{enumerate}
\end{theo}

\begin{remark}
The proof of Assertion~(\ref{it:MStrictlyConvex}) is somewhat simpler than Assertion~(\ref{it:StrictlyConvex}). 
However, in some cases, the first purpose of the penalty $J_n$ is precisely to solve an ill-posed problem such as in the
ridge regression (see~\cite{hoerl70}) where $M_n(\bphi)=\sum_k(y_k-\bx_k^T\bphi)^2$, $J_n(\bphi)\propto\|\bphi\|^2$ and the regression matrix
$\bX_n=[\bx_1\;\;\dots\;\;\bx_n]^T$ is not full rank. Thus $J_n$ is strictly convex and $M_n$ is not, in which case
Assertion~(\ref{it:StrictlyConvex}) can be useful.   
\end{remark}

\subsection{Functional central limit theorem}
\label{sec:penMestCLT}

Some general conditions for proving 
$\sqrt{n}$ asymptotic normality for M-estimators rely on the so called stochastic
differentiability condition introduced in \cite{POL85}.  
They exploit the idea introduced in \cite{huber:1967} of using strong differentiability conditions on the limit contrast 
function rather than on the contrast process. Moreover it is explained in \cite{POL85} how the empirical process theory can
be used to prove the   stochastic differentiability condition. 
Extensions of these ideas can be found in~\cite{vandervaart:wellner:1996}. 

In \cite{POL85}, Pollard proves the asymptotic normality of M-estimators
based on a contrast process of the form 
\begin{equation}
  \label{eq:contrastIIDMest}
M_n(\bphi)=n^{-1}\sum_{k=1}^n g(\xi_k,\bphi)=P_n g(\cdot,\bphi) \; ,
\end{equation} 
where $(\xi_k)$ is a sequence of $\Xset$-valued random variables and $g$ is a $\Xset\times\Rset^p$ function 
satisfying the following Taylor expansion around a given point $\bbeta\in\Rset^p$,
  \begin{equation}
    \label{eq:TaylorOnG}
    g(x,\bphi)=g(x,\bbeta)+(\bphi-\bbeta)^T\Delta(x) +\|\bphi-\bbeta\|\,r(x,\bphi)\;.
  \end{equation}

We will show that if the $\sqrt{n}$ asymptotic normality conditions in
\cite{POL85} are verified and if the penalty satisfies mild asymptotic conditions then the penalized
version of the M-estimator satisfies a  CLT  similar to the CLT in~\cite{KNI00} for the mean square criterion.   
Moreover this CLT applies to the regularization path in a functional sense.
% A quite different approach is to use convex assumptions on the contrast process, for which
% some general results on the  asymptotic normality of M-estimators are also available \cite{haberman:1989}, \cite{niemiro:1992},
% \cite{geyer:1996} , \cite{hjort:pollard:1993}. 

Let us recall Pollard's conditions that we will use on the contrast process $M_n$ defined by~(\ref{eq:contrastIIDMest})
and~(\ref{eq:TaylorOnG}). 
\begin{enumerate}[(P-1)]
\item\label{it:Pollardiid} $(\xi_k)$ is a sequence of i.i.d. random variables with distribution $P$;
\item\label{it:PollardMder} the function $M(\bphi)=Pg(\cdot,\bphi)$ has a nonsingular second derivative $\Gamma$ at $\bbeta\in\Rset^p$;
\item\label{it:PollardDelta} $P\|\Delta\|^2<\infty$ and $P\Delta=0$;
\item\label{it:SDC} the stochastic differentiability condition holds on $r$, that is, for any sequence of positive r.v. $(r_n)$ such that $r_n\cp0$,
  \begin{equation}
    \label{eq:SDC}
    \sup_{\|\bphi-\bbeta\|\leq r_n}\frac{\left|\nu_n\,r(\cdot,\bphi)\right|}{1+\sqrt{n}\|\bphi-\bbeta\|}\cp 0 \;.
  \end{equation}
\end{enumerate}
Here we used the notations, standard in the empirical process literature, $Pf$, $P_nf$ and $\nu_nf$ for $\int f \rmd P$,
$n^{-1}\sum_{k=1}^nf(\xi_k)$ and $\sqrt{n}(P_nf-Pf)$, respectively.   
Theorem~\ref{theo:PollardWithPen} below provides a central limit theorem for the regularization path defined on the penalized 
contrast~(\ref{eq:DefPenContrast}) when $M_n$ satisfies Pollard's conditions~(P-\ref{it:Pollardiid})--(P-\ref{it:SDC}) with 
some mild conditions on the penalty $J_n$.

\begin{theo}
\label{theo:PollardWithPen}
Let  $\bPhi=\Rset^p$, $p\geq1$ and $\paramset$ be a compact subset of $[0,\infty)$.
Define  $\Lambda_n$ as in~(\ref{eq:DefPenContrast}), where $M_n$ is defined by~(\ref{eq:contrastIIDMest}) and satisfies
Pollard's conditions~(P-\ref{it:Pollardiid})--(P-\ref{it:SDC})
and $J_n$ is a sequence of deterministic non-negative functions defined on $\Rset^p$.
Further assume that there exists a positive constant $C$ such that
\begin{equation}
  \label{eq:condJn}
n\;\left|J_{n}(\bphi)-J_{n}(\bbeta)\right|\leq C\, (1+\sqrt{n}\,||\bphi-\bbeta||)\quad\text{for}\quad\|\bphi-\bbeta\|\leq1\;,
\end{equation}
and, for any compact $K\subset\Rset^p$, 
\begin{equation}
  \label{eq:penAsymp}
\sup_{\bphi\in K}  \left|n\;J_{n}(\bbeta+n^{-1/2}\bphi)-n\;J_n(\bbeta)-J_{\infty}(\bphi)\right| \to 0 \;,
\end{equation}
where $J_{\infty}$ is a real-valued function on $\bPhi$.
Let $\{\hbbeta_n,\;\param\in\paramset\}$ be a sequence of $\bPhi$-valued processes 
satisfying~(\ref{eq:DefMestimaPen}) and such that the uniform $P^*$-consistency~(\ref{eq:ConvUnifMestPen}) holds.
Let $W$ be a centered Gaussian $p$-dimensional vector with covariance $P(\Delta\Delta^T)$ and define 
\begin{equation}
  \label{eq:pencontrastlimDefPolard}
  \pencontrastlim(\bphi,\param)=W^T\phi+\bphi^T\Gamma\bphi+ \param J_{\infty}(\bphi) \; .
\end{equation}
Finally assume that there exists a $\bPhi$-valued process 
$\{\hbbetadif(\param),\;\param\in\paramset\}$ such that
Conditions~(\ref{it:isolatedMinimumlimitCLT}) and~(\ref{it:limitTight}) in
Theorem~\ref{theo:argmaxParam} hold.
Then there is a version of $\hbbetadif$ in $\ell^\infty(\paramset,\Rset^p)$ and 
\begin{equation}
  \label{eq:PollardCLTStochDiff}
\sqrt{n}(\hbbeta_n-\bbeta)\wc\hbbetadif  \;.
\end{equation}
\end{theo}

The following lemma shows that the penalties considered in~\cite{KNI00} satisfy Conditions~(\ref{eq:condJn})
and~(\ref{eq:penAsymp}). 

\begin{lem}\label{lem:KFpenalties}
Let $\gamma>0$ and define, for all $\bphi=(\phi_1,\dots,\phi_p)\in\Rset^p$,
\begin{align}
  \label{eq:penaltiesDef}
  J_n^{(\gamma)}(\bphi)=n^{(1\wedge \gamma)/2-1}\sum_{k=1}^p |\phi_k|^\gamma \; .
\end{align}
Then for any $\bbeta\in\Rset^p$, there exists $C>0$ such that, for all $\bphi\in\Rset^p$,
\begin{equation}
  \label{eq:boundJnKF}
  n\left|J_n^{(\gamma)}(\bphi)-J_n^{(\gamma)}(\bbeta)\right|
\leq C\,\left(1+\sqrt{n}\|\bphi-\bbeta\|+\sqrt{n}\|\bphi-\bbeta\|^{1\vee\gamma} \right)\; ,
\end{equation}
and, for any compact $K\subset\Rset^p$,
\begin{equation}
  \label{eq:convJnKF}
\sup_{\bphi\in K}\left|nJ_n^{(\gamma)}(\bbeta+n^{-1/2}\bphi)-nJ_n^{(\gamma)}(\bbeta)-J_\infty^{(\gamma)}(\bphi)\right|\to 0\;,
\end{equation}
where
\begin{equation}
  \label{eq:JinftyKF}
J_\infty^{(\gamma)}(\bphi)=
\begin{cases}
\sum_{j=1}^p |\phi_j|^\gamma {\1}_{\{\beta_j = 0\}}
&\text{ if $\gamma<1$}\\
\sum_{j=1}^p \left\{\phi_j \sgn\left(\beta_j\right) {\1}_{\{\beta_j \neq 0\}} + |\phi_j|
          {\1}_{\{\beta_j = 0\}}\right\} 
&\text{ if $\gamma=1$}\\
\gamma\sum_{j=1}^p \phi_j \sgn\left(\beta_j\right) |\beta_j|^{\gamma-1} {\1}_{\{\beta_j \neq 0\}}
&\text{ if $\gamma>1$.}
\end{cases}
\end{equation}
\end{lem}
\begin{remark}
The limit penalties in~(\ref{eq:JinftyKF}) correspond to those in Theorems~2 and~3 in~\cite{KNI00}, except for the
multiplicative constant $\gamma$ in the case $\gamma>1$, which seems to have been forgotten in~\cite{KNI00}.
\end{remark}

\section{Pathwise M-estimation}
\label{sec:MestimCLT}

It turns out that the specific form of the contrast $\pencontrast$ in~(\ref{eq:DefPenContrast}) is not fundamental for
the basic arguments yielding the consistency and the CLT in Theorems~\ref{theo:uniform-consistency}
and~\ref{theo:PollardWithPen}, respectively. Here we provide results formulated in the more general form where
$\hbbeta_n(\param)$ is a near minimizer of $\pencontrast(\cdot,\param)$ for all $\param\in\paramset$.
Moreover the true parameter $\bbeta$ itself is defined as a map on $\paramset$, with  $\bbeta(\paramset)$
defined as the minimizer of $\pencontrastlim(\cdot,\param)$ for all $\param\in\paramset$. We refer this general situation as  
\emph{pathwise M-estimation}.

\subsection{Uniform consistency}
Theorem~\ref{theo:uniform-consistency} is obtained by applying the following general result on pathwise M-estimators. 

\begin{prop} \label{prop:uniform-consistency}
Let $\bPhi$ be a subset of a metric space endowed with the metric $d$ and $\paramset$ be any set. 
Let $\Lambda$ be a real-valued function defined on $\bPhi\times\paramset$,
$\{\pencontrast(\bphi,\param),\;\bphi\in\bPhi,\param\in\paramset\}$ be a sequence of real-valued processes, $\bbeta$ be a
$\paramset\to\bPhi$ map and $\{\hbbeta_n(\param),\;\param\in\paramset\}$ be a sequence of $\bPhi$-valued processes 
such that  
  \begin{enumerate}[(i)]
  \item\label{it:consUnfiLowBoundParam} 
$\displaystyle\sup_{\bphi\in\bPhi}\sup_{\param\in\paramset}
\left\{\Lambda(\bphi;\param)-\pencontrast(\bphi,\param)\right\}_+\cpout 0$;
  \item\label{it:SimpleConvParam} 
$\displaystyle\sup_{\param\in\paramset}\left|\pencontrast(\bbeta(\param),\param)-\Lambda(\bbeta(\param),\param)\right|\cpout0$;
  \item \label{it:IsolMinParam} For all $\epsilon>0$,
    $$
\inf_{\param\in\paramset}\left[\inf\{\Lambda(\bphi;\param)~:~\bphi\in\bPhi,\;d(\bphi,\bbeta(\param))\geq\epsilon\}-\Lambda(\bbeta(\param),\param)\right]>0\;; 
$$
  \item \label{it:DefMestimaParam} 
    $\displaystyle\sup_{\param\in\paramset}\displaystyle\left\{\pencontrast(\hbbeta_n(\param),\param)-\pencontrast(\bbeta(\param),\param)\right\}_+\cpout 0$.
  \end{enumerate}
Then, $\hbbeta_n(\param)$ converges to $\bbeta(\param)$ uniformly in $\param\in\paramset$, in $P^\ast$--probability, that is,
\begin{equation}
  \label{eq:ConvUnifMestPenParam}
  \sup_{\param\in\paramset}d(\hbbeta_n(\param),\bbeta(\param))\cpout 0\;.
\end{equation}
\end{prop}

\subsection{Functional central limit theorem}

We now extend the setting of  \cite{POL85} to pathwise M-estimation. First we 
obtain the $\sqrt{n}$-rate of convergence in the sup norm; second we apply Theorem~\ref{theo:argmaxParam}
to obtain a functional CLT for pathwise M-estimators. Theorem~\ref{theo:PollardWithPen} is a direct application of this
result  in the context of penalized M-estimation.

\begin{prop}  \label{prop:vitesse}
Let $\bPhi$ be a subset of a metric space endowed with the metric $d$ and $\paramset$ be any set. 
Let
$\{\pencontrast(\bphi,\param),\;\bphi\in\bPhi,\param\in\paramset\}$ be a sequence of real-valued processes, $\bbeta$ be a
$\paramset\to\bPhi$ map and $\{\hbbeta_n(\param),\;\param\in\paramset\}$ be a sequence of $\bPhi$-valued processes 
such that
\begin{equation}
\label{eq:DefMestimaParamVitesse} 
\sup_{\param\in\paramset}
\left\{\pencontrast(\hbbeta_n(\param),\param)-\pencontrast(\bbeta(\param),\param)\right\}_+
=O_{P^*}\left(n^{-1}\right)\;,   
\end{equation}
and the uniform $P^*$-consistency~(\ref{eq:ConvUnifMestPenParam}) holds.
Assume that we have the following decomposition of the contrast process, 
\begin{equation}
\label{eq:ContrasteDecompVitesse}
\pencontrast(\bphi,\param)-\pencontrast(\bbeta(\param),\param)= 
G_n(\bphi,\param)+
H(\bphi,\param)+
d(\bphi,\bbeta(\param))\;R_n(\bphi,\param) \;, 
\end{equation}
where $G_n$, $H$ and $R_n$ satisfy
\begin{enumerate}[(i)]
  \item\label{it:Gncond} $\{G_n(\bphi,\param),\;\bphi\in\bPhi,\param\in\paramset\}$ is  a sequence of real-valued processes
    such that
\begin{equation}
\label{eq:Gncond}
\sup_{\bphi\in\bPhi}\sup_{\param\in\paramset}
\frac{n\;\left|G_n(\bphi,\param)\right|}{1+\sqrt{n}\,d(\bphi,\bbeta(\param))}=O_{P^*}(1)\;;
\end{equation}
\item\label{it:Hcond} $H$ is a real-valued function defined on $\bPhi\times\paramset$ such that there exists $\epsilon>0$ for
    which 
\begin{equation}
\label{eq:Hcond}
\inf_{\param\in\paramset}\inf \left\{\frac{H(\bphi,\param)}{d^2(\bphi,\bbeta(\param))}~:~\bphi\in\bPhi,
  \;d(\bphi,\bbeta(\param))\leq\epsilon \right\} >0\;;
\end{equation}
\item \label{it:Rncond}  $\{R_n(\bphi,\param),\;\bphi\in\bPhi,\param\in\paramset\}$ is  a sequence of real-valued processes
    such that, for any positive random sequence $(r_n)$ converging to 0 in $P^*$--probability, 
\begin{equation}
\label{eq:Rncond}
\sup_{\param\in\paramset}\sup\left\{\left|R_n(\bphi,\param)\right|~;~\bphi\in\bPhi,
  \;d(\bphi,\bbeta(\param))\leq r_n \right\} = o_{P^*}(r_n)+O_{P^*}(n^{-1/2}) \; .
\end{equation}
\end{enumerate}
Then, $\hbbeta_n(\param)$ converges to $\bbeta(\param)$ uniformly in $\param\in\paramset$, in $P^\ast$--probability, with rate at least
$\sqrt{n}$, that is,
\begin{equation}
  \label{eq:ConvUnifMestPenParamVitesse}
  \sqrt{n}\;\sup_{\param\in\paramset}d(\hbbeta_n(\param),\bbeta(\param)) = O_{P^*}(1) \;.
\end{equation}
\end{prop}

Applying Proposition~\ref{prop:vitesse} and Theorem~\ref{theo:argmaxParam}, we get the following result.

\begin{theo}  \label{theo:CLTStochDiff}
Let  $\bPhi=\Rset^p$, $p\geq1$, and $\paramset$ be any set. 
Let
$\{\pencontrast(\bphi,\param),\;\bphi\in\bPhi,\param\in\paramset\}$ be a sequence of real-valued processes, $\bbeta$ be a
$\paramset\to\bPhi$ map and $\{\hbbeta_n(\param),\;\param\in\paramset\}$ be a sequence of $\bPhi$-valued processes 
such that
\begin{equation}
\label{eq:DefMestimaParamVitesseS} 
\sup_{\param\in\paramset}
\left\{\pencontrast(\hbbeta_n(\param),\param)-\pencontrast(\bbeta(\param),\param)\right\}_+
=o_{P^*}\left(n^{-1}\right)\;,   
\end{equation}
and the uniform $P^*$-consistency~(\ref{eq:ConvUnifMestPenParam}) holds.
Assume that the decomposition~(\ref{eq:ContrasteDecompVitesse}) of the contrast process holds
where $G_n$, $H$ and $R_n$ satisfy:
\begin{enumerate}[(i)]
  \item\label{it:GncondS} $\{G_n(\bphi,\param),\;\bphi\in\bPhi,\param\in\paramset\}$ is  a sequence of real-valued processes
    satisfying~(\ref{eq:Gncond});
\item\label{it:HcondS} $H$ is a real-valued function defined on $\bPhi\times\paramset$ and there exists a function $\Gamma$
  defined on $\paramset$ and taking values in the set of non-negative symmetric  
$p\times p$ matrices such that, denoting by $\lambda_{\min}(\Gamma(\param))$  and $\lambda_{\max}(\Gamma(\param))$ the
smallest and largest eigenvalues  of $\Gamma(\param)$, 
\begin{equation}
\label{eq:HcondS1}
0<\inf\{\lambda_{\min}(\Gamma(\param)),\;\param\in\paramset\}<\sup\{\lambda_{\max}(\Gamma(\param)),\;\param\in\paramset\}<\infty\;,
\end{equation}
and, as $\bphi\to\bbeta$ in $\ell^\infty(\paramset,\Rset^p)$,
\begin{equation}
\label{eq:HcondS2}
\left\|H(\bphi(\cdot),\cdot)-
  (\bphi-\bbeta)^T\Gamma(\bphi-\bbeta)\right\|_{\paramset}=o\left(\|\bphi-\bbeta\|^2_{\paramset}\right)\;;
\end{equation}
\item \label{it:RncondS}  $\{R_n(\bphi,\param),\;\bphi\in\bPhi,\param\in\paramset\}$ is  a sequence of real-valued processes
    such that, for any positive random sequence $(r_n)$ converging to 0 in $P^*$--probability, 
 \begin{equation}
\label{eq:RncondS}
\sup_{\param\in\paramset}\sup\left\{\left|R_n(\bphi,\param)\right|~;~\bphi\in\bPhi,
  \;d(\bphi,\bbeta(\param))\leq r_n \right\} = o_{P^*}(r_n)+o_{P^*}(n^{-1/2}) \; .
\end{equation}
\end{enumerate}
Let us further define 
\begin{equation}
  \label{eq:pencontrastnormLinPartDef}
  \widehat{G}_n(\bphi,\param)=nG_n\left(\bbeta(\param)+n^{-1/2}\bphi,\param\right )\;,
\end{equation}
and assume that there exists a real-valued process $\{G(\bphi,\param),\;\bphi\in\bPhi,\;\param\in\paramset\}$
such that, for any compact $K\subset\bPhi$, $G$ is tight in
$\ell^\infty(K\times\paramset,\Rset^p)$ and $\widehat{G}_n\wc G$ in
$\ell^\infty(K\times\paramset,\Rset^p)$. Define
\begin{equation}
  \label{eq:pencontrastlimDef}
  \pencontrastlim(\bphi,\param)=G\left(\bphi,\param\right )+
\bphi^T\Gamma(\param)\bphi\;,
\end{equation}
and assume that there exists a $\bPhi$-valued process 
$\{\hbbetadif(\param),\;\param\in\paramset\}$ such that
Conditions~(\ref{it:isolatedMinimumlimitCLT}) and~(\ref{it:limitTight}) in
Theorem~\ref{theo:argmaxParam} hold.
Then there is a version of $\hbbetadif$ in $\ell^\infty(\paramset,\Rset^p)$ and 
\begin{equation}
  \label{eq:CLTStochDiff}
\sqrt{n}(\hbbeta_n-\bbeta)\wc\hbbetadif  \;.
\end{equation}
\end{theo}
\begin{remark}
Observe that Eq.~(\ref{eq:DefMestimaParamVitesseS}) is a strengthened version of~(\ref{eq:ConvUnifMestPenParamVitesse})
and that~(\ref{eq:HcondS1}) and~(\ref{eq:HcondS2}) imply~(\ref{eq:Hcond}). Hence
Conditions~(\ref{it:GncondS})--(\ref{it:RncondS}) in Theorem~\ref{theo:CLTStochDiff} imply 
Conditions~(\ref{it:Gncond})--(\ref{it:Rncond}) in Proposition~\ref{prop:vitesse}.
\end{remark}

\section{Examples}
\label{sec:OtherExamples}

The uniform consistency and a functional central limit theorem for the lasso regularization path are respectively given 
in Theorems~\ref{theo:lassoConsistence} and~\ref{theo:lassoCLT}. Theorems~\ref{theo:ConvexConsistence}
and~\ref{theo:PollardWithPen} allow many extensions, some examples of which are given in this section. 
In  \cite{POL85}, a wide variety of models and functions $g$ are shown to satisfy Conditions
(P-\ref{it:Pollardiid})--(P-\ref{it:SDC}). These conditions apply for the general linear model (GLM) as this model
satisfies the pointwise assumptions of  \cite[Section~4]{POL85} (provided some moment conditions). They also apply for the
least absolute deviation (LAD) criterion, see Example~8 in \cite[Section~6]{POL85} (provided again some moment conditions on
the model). We briefly write the corresponding results in these two cases as examples of applications of
Theorem~\ref{theo:PollardWithPen}. Uniform consistencies for both examples are obtained as applications of
Theorem~\ref{theo:ConvexConsistence}, since in these cases $M_n$ is convex. 
For these two examples, we consider the $\ell^1$ and $\ell^2$ penalties. They fit the conditions of
Theorem~\ref{theo:PollardWithPen} as they satisfy~(\ref{eq:condJn}) and~(\ref{eq:penAsymp}) by
Lemma~\ref{lem:KFpenalties}. Observe however that the function $J_\infty$ in Lemma~\ref{lem:KFpenalties} depends on the chosen
penalty and thus so does the limit $\hbbetadif$ in~(\ref{eq:PollardCLTStochDiff}).
We conclude this section with a discussion on the Akaike information criterion (AIC), which corresponds to a $\ell^0$
penalty. 

\subsection{$\ell^1$--penalized GLM}
Consider a canonical exponential family of density
$$
p(y|\theta)=h(y)\exp\{y\theta-b(\theta)\} \; ,
$$
with respect to a dominating measure $\mu$. The function $b$, sometimes called the log-repartition function, is given by 
$$
b(\theta)=\log \int h(y) \exp\{y\theta\}\mu(\rmd y) \; ,
$$
and thus is strictly convex and infinitely differentiable. In a GLM, one observes a sequence of
i.i.d. $\Rset\times\Rset^p$-valued r.v.'s $(y_k,\bx_k),\;k=1,\dots,n$, where $y_k$ have conditional density
$p(\cdot|\bx_k^T\bbeta)$, given $\bx_k$, with $\bbeta\in\Rset^p$ denoting the unknown parameter of interest.  In this context,  the
non-penalized contrast process is given by the negated log-likelihood
$$
M_n(\bphi) = n^{-1} \sum_{k=1}^n g((\bx_k,y_k),\bphi) \; ,
$$
where $g((\bx,y),\bphi)=-y\bx^T\bphi + b(\bx^T\bphi)$. Using that $g$ is convex and
smooth, and assuming some appropriate moment conditions on $\bx_1$  for obtaining Pollard's
conditions~(P-\ref{it:Pollardiid})--(P-\ref{it:SDC}),  we get the uniform consistency and a 
functional CLT on the regularization path $\hbbeta_n(\param)$ defined as the minimizer of~(\ref{eq:DefPenContrast}) with
$J_n(\bphi)=n^{-1/2}\sum_{i=1}^p|\phi_i|$ (this is the $\ell^1$ penalty $J^{(1)}_n$ defined in~(\ref{eq:penaltiesDef})). In
particular, for any $L>0$, 
$$
\sqrt{n}(\hbbeta_n-\bbeta)\wc\hbbetadif  \; \textrm{ in } \ell^\infty([0,L],\Rset^p) \;,
$$
where the limit $\hbbetadif$ is defined as in the lasso case as the minimizer of~(\ref{eq:ContrasteLimite}) with 
$C=\PE[b''(\bx_1^T\bbeta)\bx_1\bx_1^T]$ (assumed positive-definite) and $U\sim\mathcal{N}(0,C)$.
The numerical computation of $\hbbeta_n(\param)$ can be processed as proposed in~\cite{PAR06}. 

\subsection{$\ell^1$ and $\ell^2$--penalized LAD} Given  a sequence of
$\Rset\times\Rset^p$-valued r.v.'s $(y_k,\bx_k),\;k=1,\dots,n$, the LAD criterion is defined as
$$
M_n(\bphi) = n^{-1} \sum_{k=1}^n |y_k-\bx_k^T\bphi| \; .
$$
It can be used to estimate the parameter $\bbeta\in\Rset^p$ of a linear regression model $y_k=\bx_k^T\bbeta+\varepsilon_k$, with
$(\varepsilon_k)$ and $(\bx_k)$ two independent sequence of i.i.d. r.v.'s.
This contrast process is an alternative to the mean square criterion, resulting in an estimator less sensitive to the
presence of outliers (for $\bx_k=1$, the minimizer of $M_n$ is 
the sample median). In contrast to the previous case, the contrast is not smooth, since the first derivative is
discontinuous. However, as shown \textit{e.g.} in \cite{POL85}, the minimizer of this contrast is asymptotically normal, provided some
moment conditions and that
$$
G(\bphi)=\PE\left[\left|\varepsilon_1+\bx_1^T(\bbeta-\bphi)\right|\right]
$$
has a non-singular second derivative at $\bphi=\bbeta$. 
Observe that
$$
G(\bphi)=\PE\left[\bx_1^T(\bbeta-\bphi)+2\int_{0}^{\bx_1^T(\bphi-\bbeta)}F(s)\;\rmd s\right]\;,
$$ 
where $F$ denotes the cumulative distribution function of $\varepsilon_1$. Thus,
if $\varepsilon_1$ is distributed from a continuous density $f$, the 
second derivative of $G$ at $\bbeta$ is $\Gamma=2f(0)\PE\left[\bx_1\bx_1^T\right]$.
Because the LAD criterion uses the $\ell^1$ error function, the $\ell^2$
penalty $J_n(\bphi)=n^{-1/2}\sum_{i=1}^p\phi_i^2$ could seem more reasonable. On the contrary
Theorem~\ref{theo:PollardWithPen} suggests that using an $\ell^1$ error 
function contrast does not modify the asymptotic distribution of the regularization path, only the choice of the penalty
does. In other words, the regularization path of the $\ell^1$ and $\ell^2$--penalized LAD  has similar asymptotic distributions as 
the lasso and the ridge regression, respectively.
Let us now precise the limit distribution of the regularization path $\hbbeta_n(\param)$ defined as the minimizer of~(\ref{eq:DefPenContrast}) with
$J_n(\bphi)=n^{-1/2}\sum_{i=1}^p|\phi_i|$ and $J_n(\bphi)=n^{-1/2}\sum_{i=1}^p\phi_i^2$ respectively
(these are the $\ell^1$ and $\ell^2$ penalty $J^{(1)}_n$ and  $J^{(2)}_n$ defined in~(\ref{eq:penaltiesDef})). 
Under appropriate moment conditions on $(\varepsilon_1,\bx_1)$ implying Pollard's
conditions~(P-\ref{it:Pollardiid})--(P-\ref{it:SDC}) (in particular $\PE[\sgn(\varepsilon_1)]=0$, $\PE[\|\bx_1\|^2]<\infty$
so that $\PE[\Delta]=0$,  $\PE[\|\Delta\|^2]<\infty$ and $G$ is minimized at $\bphi=\bbeta$), one has, for any $L>0$,  
$$
\sqrt{n}(\hbbeta_n-\bbeta)\wc\hbbetadif  \; \textrm{ in } \ell^\infty([0,L],\Rset^p) \;,
$$
where the limit $\hbbetadif$ is defined as the minimizer of~(\ref{eq:pencontrastlimDefPolard}) where $\Gamma$ is the
(non-singular) second derivative of $G$ at $\bphi=\bbeta$,  $W\sim\mathcal{N}(0,\PE[\bx_1\bx_1^T])$ and 
$J_\infty$ depends on the penalty. Namely, for the $\ell^1$ penalty, one has $J_\infty=J_\infty^{(1)}$ and 
for the $\ell^2$ penalty, one has $J_\infty=J_\infty^{(2)}$, where $J_\infty^{(\gamma)}$ is defined by~(\ref{eq:JinftyKF}). 

\subsection{Akaike information criterion and the $\ell^0$ penalty}
Consider a parametric family of densities $\{p_{\bphi}\,,\,\bphi\in\bPhi\}$ defined on $\Xset^n$ for modelling the
distribution of the observations $\xi_1,\dots,\xi_n$.
The Akaike  information criterion (AIC) was proposed in \cite{akaike:1973} as the negated log-likelihood
criterion penalized by the dimension of the parameter. It can be defined (up to a multiplicative factor which does not change
its minimizer) as  
$$
\mathrm{AIC}(\bphi)=\pencontrast(\bphi,1)\;,
$$ 
where $\pencontrast$ is defined by~(\ref{eq:DefPenContrast}) with $M_n(\bphi)=-n^{-1}\log p_{\bphi}(\xi_1,\dots,\xi_n)$ and 
$$
J_n^{(0)}(\bphi)=n^{-1}\#\left\{k\,:\,\phi_k\neq0\right\} \;,
$$
where $\# A$ denotes the cardinality of the set $A$. 
We note that it corresponds to a $\ell^0$ penalty, that is, to $\gamma=0$ in~(\ref{eq:penaltiesDef}) although this case is
not considered in \cite{KNI00}.  It is not usually assumed that $\bPhi$ is finite-dimensional in the
presentation of the AIC. However, in practice,  
the minimization of $\mathrm{AIC}(\bphi)$ requires numerically minimizing $M_n(\bphi)$ for each possible submodel, which
corresponds to a given value of the sequence 
$(\1(\phi_k\neq0) )_{k\geq1}$. This makes sense only in a finite-dimensional setting,  $\bPhi\subseteq\Rset^p$, with $p$
not too large (say $p\leq 15$) since $2^p$ numerical minimizations of $M_n$ are then necessary.

Observe that, for any fixed $\bphi\in\Rset^p$ we have $nJ_n^{(0)}(\bphi)\leq p$ and, for any $\bbeta\in\Rset^p$ and any
$r>0$, we have, for $n$ large enough,
$$
nJ_n^{(0)}(\bbeta+n^{-1/2}\bphi)-nJ_n^{(0)}(\bbeta)=J_\infty^{(0)}(\bphi)
\quad\text{for all}\quad\|\bphi\|\leq r\;,
$$
where
$$
J_\infty^{(0)}(\bphi)=\sum_{k=1}^p\1(\beta_k=0\text{ and }\phi_k\neq0)\;.
$$
It follows that the contrast $J_n^{(0)}$ satisfies the assumptions~(\ref{eq:condJn}) and~(\ref{eq:penAsymp}) in
Theorem~\ref{theo:PollardWithPen} and thus we may apply this result to obtain the limit behavior of the minimizer of the AIC 
in the i.i.d. case, that is, when 
$$
M_n(\bphi)=P_n g(\cdot,\bphi)\quad\text{with}\quad g(x,\bphi)=-\log p_{\bphi}(x) \;.
$$
Here, $p_{\bphi}$ denotes the density of one observation in the parametric family $\{p_{\bphi}\,,\,\bphi\in\bPhi\}$. 
Suppose that this model satisfy Assumption~\ref{assum:uniform-consistency} and
the Pollard's conditions (P-\ref{it:Pollardiid})--(P-\ref{it:SDC}) with $\bbeta$ denoting the true parameter and with
$\Gamma=P(\Delta\Delta^T)$ equal to the Fisher information matrix at parameter $\bbeta$.
We may thus apply
Theorem~\ref{theo:uniform-consistency} and Theorem~\ref{theo:PollardWithPen} successively to the minimizing sequence 
$$
\hbbeta_n=\argmin_{\bphi\in\bPhi}\mathrm{AIC}(\bphi)\;.
$$
We obtain $\hbbeta_n\wc\bbeta$ and $\sqrt{n}(\hbbeta_n-\bbeta)\wc\hbbetadif$, where $\hbbetadif$ is defined as the minimizer 
of~(\ref{eq:pencontrastlimDefPolard}) with $J_\infty=J_\infty^{(0)}$ and $\param=1$.
Observe that, in the limit penalty $J_\infty^{(0)}$, only the vanishing coordinates of the true parameter $\bbeta$ are penalized. 
In other words, for a coordinate $k$ such that $\beta_k=0$ and only for such a coordinate, we have $\hbbetadif_k=0$ with
positive probability. This property highlights the (well known) ability of the AIC criterion to correctly select the correct
model. 

Finally we note that the AIC can easily be extended to a collection of contrast 
$\pencontrast(\bphi,\param)$, where $\param$ is a positive penalty weight (the case $\param=1$ corresponding to the standard AIC). 
The solution path $\hbbeta_n(\param)$ with minimizes $\pencontrast(\bphi,\param)$ for all $\param>0$ is not more difficult to 
compute than $\hbbeta_n(1)$ once one has minimized $M_n(\bphi)$ for the $2^p$ possible submodels.
One easily sees that the solution path is piece-wise constant with multiple solutions at the discontinuities.
Multiple solutions for a finite set of penalty weight $\param$ are also present 
in the limit contrast~(\ref{eq:pencontrastlimDefPolard}) with $J_\infty=J_\infty^{(0)}$.
One can show that there exists almost surely a unique minimizer $\hbbetadif(\param)$ of the limit contrast
$\pencontrastlim(\cdot,\param)$ for all $\param\in\paramset$ if and only if the closure of the set $\paramset$ has zero
Lebesgue measure. In the latter case, one also has that $\hbbetadif$ satisfies Condition~(\ref{it:isolatedMinimumlimitCLT}) in 
Theorem~\ref{theo:argmaxParam}. This non-uniqueness problem of the minimizer of the limit contrast
did not appear in the previous examples because for both $\ell^1$ and $\ell^2$ penalties, the limit contrast was strictly 
convex. This is non-longer true for the $\ell^0$ penalty so that the convergence
$\sqrt{n}(\hbbeta_n-\bbeta)\wc\hbbetadif$ cannot hold in a functional sense in this case. 
Nevertheless, the convergence continues
to hold  for the $\ell^0$ penalty in the sense of the finite-dimensional convergence because, for a given finite number of
penalty weights $\param$, there is a unique minimizer $\hbbetadif(\param)$ of $\pencontrastlim(\cdot,\param)$ almost
surely. 

\section{Conclusion}
\label{sec:conclusion}
We extended the works of  Knight and Fu (2000) in several ways by showing that the asymptotic distribution 
that they exhibited for the penalized least squared continues to hold 1) for the solution path in a functional sense 2) for a 
wide variety of contrasts extending the least squares case. We provided several examples of interest.
An interesting feature of penalized estimation is that the form of the limit distribution of the regularization path
only depends on the penalty since for any \emph{standard} contrast, it is given as the path
minimizing~(\ref{eq:pencontrastlimDefPolard}) with $J_\infty$ only depending on the penalty. 
The marginal limit distribution is discussed in   Knight and Fu (2000) for $\ell^\gamma$ penalties with $\gamma>0$.
As pointed out in this reference, a particular feature of $\ell^1$ penalty is that the limit distribution is compatible with
model selection properties but introduce an additional bias on the non-vanishing components.
We have shown that the model selection property is preserved by the $\ell^0$ penalty, without introducing an additional  bias
on the non-vanishing components. However, the  $\ell^0$ penalty is much less numerically tractable for a large dimension of
the parameter space and the central limit theorem on the solution path only holds in a finite-dimensional sense. 
This latter result were derived in Section~\ref{sec:OtherExamples} for the AIC in the i.i.d. case. A similar analysis
can clearly be carried out for the AIC applied to time series models or for Mallow's $C_p$ criterion. 

\section*{Acknowledgements}

We would like to thank the referees for their valuable comments.

\bibliographystyle{plainnat}
\bibliography{local}

\subsection*{Appendix: detailed proofs}

\begin{proof}[Proof of Theorem~\ref{theo:argmaxParam} in the case~(C-\ref{it:LinftyConvinRp})]
Recall that, in the  case~(C-\ref{it:LinftyConvinRp}), we set $\bPhi=\Rset^p$ and   $\calD=\ell^\infty(\paramset,\Rset^p)$. 
By Theorem~1.5.4 in~\cite{vandervaart:wellner:1996}, to show that $\hbbetadif$ admits a version in $\calD$ with
$\hbbetadif_n\wc\hbbetadif$ in $\calD$, it is sufficient to 
show that the finite-dimensional distributions of $\hbbetadif_n$ converge to those of $\hbbetadif$ and that
$(\hbbetadif_n)$ is asymptotically tight. The convergence of the finite-dimensional distributions  follows from the case 
(C-\ref{it:finite-dim-case}) that we already proved. Hence to conclude the proof in the case~(C-\ref{it:LinftyConvinRp}), it
only remains to show that $(\hbbetadif_n)$ is asymptotically tight. 
In the following we show that this uniform tightness is inherited from that of $(\pencontrastnorm)$ in
$\ell^\infty(K\times\paramset)$. Asymptotic tightness follows from an equicontinuity criterion. The proof has now two steps.
In Step~1, we construct a metric $\tilde{\rho}$ on $\paramset$ based on a metric $\rho$ that makes $\pencontrastnorm$
asymptotically  uniformly equicontinuous. In Step~2 we use the metric $\tilde{\rho}$ to prove an equicontinuity criterion for
$\hbbetadif_n$. 

\noindent\textbf{Step~1.} By successively applying Lemma~1.3.8 and Theorem~1.5.7 in~\cite{vandervaart:wellner:1996},
Condition~(\ref{it:LinftyConvContrast}) 
implies that, for any compact set $K\subset\Rset^p$, $\pencontrastnorm$ is asymptotically tight in
$\ell^\infty(K\times\paramset)$ and there exists a semi-metric $\rho$ 
on $K\times\paramset$ such that $(K\times\paramset,\rho)$ is totally bounded and $\pencontrastnorm$ is asymptotically  uniformly
$\rho$-equicontinuous in probability. This means that, for any $\epsilon,\alpha>0$, there exists $\delta>0$ such that
\begin{equation}
  \label{eq:equicontContrast}
\limsup P^*\left(\sup_{(\doubleparam,\doubleparam')\in \calS_\delta(K)}|\pencontrastnorm(\doubleparam)-\pencontrastnorm(\doubleparam')|>\alpha\right)
\leq\epsilon \;,  
\end{equation}
where 
$$
\calS_\delta(K)=\left\{((\bphi,\param),(\bphi',\param'))\in (K\times\paramset)^2~:~\rho((\bphi,\param),(\bphi',\param'))<\delta\right\}\;.
$$
Clearly, the semi-metric $\rho$ can be assumed to be bounded and not to depend on the compact set $K$ without loss of
generality; in other words, a bounded semi-metric $\rho$ can be defined on $\Rset^p\times\paramset$ so that
$(\Rset^p\times\paramset,\rho)$ is totally bounded and $\pencontrastnorm$ is asymptotically  uniformly $\rho$-equicontinuous
in probability on $K\times\paramset$ for any compact set $K$. We shall use this semi-metric in the following to   
show that $\hbbetadif_n$ is asymptotically  uniformly $\tilde{\rho}$-equicontinuous in probability,
where $\tilde{\rho}$ is the semi-metric defined on $\paramset$ by
$$
\tilde{\rho}(\param,\param')=\sup_{\bphi\in\Rset^p}\rho((\bphi,\param),(\bphi,\param'))\;.
$$
By~\cite[Theorem~1.5.7]{vandervaart:wellner:1996}, the asymptotic uniform $\tilde{\rho}$-equicontinuity in probability
implies that $(\hbbetadif_n)$ is asymptotically tight.

\noindent\textbf{Step~2.} It now remains to show that  $(\hbbetadif_n)$ is asymptotically  uniformly $\tilde{\rho}$-equicontinuous in probability.
Let $\eta$ and $\epsilon$ be two arbitrarily small positive numbers.
By Conditions~(\ref{it:limitTight}) and~(\ref{it:UnifTight}), we may choose a compact $K\subset\Rset^p$ such that
\begin{equation}
  \label{eq:BBnLimCLT}
 P(B)\leq\epsilon \quad\text{and}\quad
\limsup P^*(B_n)\leq\epsilon\;,
\end{equation}
where
$$
B=\{\hbbetadif(\param)\in K\text{ for all }\param \in\paramset\}^c 
\quad\text{and}\quad B_n=\left\{\hbbetadif_n(\param)\in K\text{ for all }\param \in\paramset\right\}^c\;.
$$

Using Condition~(\ref{it:isolatedMinimumlimitCLT}), we may find $\alpha>0$ arbitrarily small such that
\begin{equation}
  \label{eq:DefEpsilonCLT}
P\left(
\inf_{\param\in\paramset}\big[
\inf\left\{\pencontrastlim(\bphi,\param)~:~\bphi\in K,\,\|\bphi-\hbbetadif(\param)\|\geq\eta/2\right\}  
- \pencontrastlim(\hbbetadif(\param),\param)\big]\leq4\alpha\right)\leq\epsilon \;.
\end{equation}
We further choose $\delta>0$ so that Inequality~(\ref{eq:equicontContrast}) holds, that is
\begin{equation}
  \label{eq:EnLimCLT}
 \limsup P^*(E_n)\leq\epsilon \;,
\end{equation}
where
$$
E_n=\left\{\sup_{(\doubleparam,\doubleparam')\in
    \calS_\delta(K)}|\pencontrastnorm(\doubleparam)-\pencontrastnorm(\doubleparam')|>\alpha\right\}\;. 
$$
Finally, Condition~(\ref{it:DefArgminFalCLT}) gives that
\begin{equation}
  \label{eq:CnLimCLT}
\limsup P^*(C_n)=0\;,
\end{equation}
where
$$
C_n=\left\{\sup_{\param\in\paramset}\left\{\pencontrastnorm(\hbbetadif_n(\param),\param)
-\inf_{\bphi\in\bPhi}\pencontrastnorm(\bphi,\param)\right\}_+ >\alpha\right\}\;.
$$
On $B_n^c$, we notice that $((\hbbetadif_n(\param'),\param),(\hbbetadif_n(\param'),\param'))\in\calS_\delta(K)$
for every  $(\param,\param')$ such that $\tilde{\rho}(\param,\param')<\delta$.
Hence, on $B_n^c\cap E_n^c$, we have
\begin{equation}
  \label{eq:BcapE_CLT}
\tilde{\rho}(\param,\param')<\delta\Rightarrow
\pencontrastnorm(\hbbetadif_n(\param'),\param)\leq\pencontrastnorm(\hbbetadif_n(\param'),\param')+\alpha \; .  
\end{equation}
Suppose for a moment that we are on the set 
$$
D_n=\left\{\sup_{\tilde{\rho}(\param,\param')<\delta}\left\|\hbbetadif_n(\param)-\hbbetadif_n(\param')\right\|>\eta
\right\}\;.
$$ 
Then we may find $(\param,\param')\in\paramset^2$ such that $\tilde{\rho}(\param,\param')<\delta$ and
$\|\hbbetadif_n(\param)-\hbbetadif_n(\param')\|>\eta$. On $C_n^c$, we further have
$\pencontrastnorm(\hbbetadif_n(\param'),\param')\leq\inf_{\bphi\in\bPhi}\pencontrastnorm(\bphi,\param') +\alpha$. 
Intersecting with $B_n^c\cap E_n^c$ and applying~(\ref{eq:BcapE_CLT}), we obtain
$$
\pencontrastnorm(\hbbetadif_n(\param'),\param)\leq\inf_{\bphi\in\bPhi}\pencontrastnorm(\bphi,\param')+2\alpha
\leq \pencontrastnorm(\hbbetadif_n(\param),\param')+2\alpha\leq
\pencontrastnorm(\hbbetadif_n(\param),\param)+3\alpha \; ,
$$
where the last inequality is obtained by exchanging $\param$ with $\param'$ in~(\ref{eq:BcapE_CLT}).
Applying again that we are on $C_n^c$, we have
$\pencontrastnorm(\hbbetadif_n(\param),\param)\leq\inf_{\bphi\in\bPhi}\pencontrastnorm(\bphi,\param) +\alpha$, 
and thus, with the last display, we get
$$
\max\left(\pencontrastnorm(\hbbetadif_n(\param),\param),\pencontrastnorm(\hbbetadif_n(\param'),\param)\right)
\leq \inf_{\bphi\in\bPhi}\pencontrastnorm(\bphi,\param) +4\alpha 
\leq \inf_{\bphi\in K}\pencontrastnorm(\bphi,\param) +4\alpha \; .
$$
Since $\|\hbbetadif_n(\param)-\hbbetadif_n(\param')\|>\eta$ and $\hbbetadif_n(\param)$ and $\hbbetadif_n(\param')$ belong to
$K$ on $B_n^c$, we just proved that $D_n\cap  C_n^c\cap B_n^c\cap E_n^c$ is included in
$$
F_n=\left\{\inf_{\param\in\paramset}\left[\inf_{(\bphi,\bphi')\in\calB_\eta(K)}
\max\left(\pencontrastnorm(\bphi,\param),\pencontrastnorm(\bphi',\param)\right)-\inf_{\bphi\in
  K}\pencontrastnorm(\bphi,\param)\right]\leq 4\alpha\right\}\;,
$$
where 
$$
\calB_\eta(K)=\left\{(\bphi,\bphi')\in K^2~:~\|\bphi-\bphi'\|>\eta\right\}\;.
$$
Using Condition~(\ref{it:LinftyConvContrast}) and the continuous mapping Theorem, we have
$\limsup P^*(F_n) \leq P(F)\;,$
where
$$
F=\left\{\inf_{\param\in\paramset}\left[\inf_{(\bphi,\bphi')\in\calB_\eta(K)}
\max\left(\pencontrastlim(\bphi,\param),\pencontrastlim(\bphi',\param)\right)-\inf_{\bphi\in
  K}\pencontrastlim(\bphi,\param)\right]\leq 4\alpha\right\}\;.
$$
Since $D_n\cap  C_n^c\cap B_n^c\cap E_n^c\subset F_n$, using~(\ref{eq:BBnLimCLT}),~(\ref{eq:EnLimCLT})
and~(\ref{eq:CnLimCLT}), we further obtain  
$$
\limsup P^*(D_n)\leq  P(F) + 2 \epsilon \; .
$$
Observe that for all $(\bphi,\bphi')\in\calB_\eta(K)$ and $\param\in\paramset$, we have $\|\bphi-\hbbetadif(\param)\|>\eta/2$  
or $\|\bphi'-\hbbetadif(\param)\|>\eta/2$. Hence, for all  $\param\in\paramset$,
$$
\inf_{(\bphi,\bphi')\in\calB_\eta(K)}\max\left(\pencontrastnorm(\bphi,\param),\pencontrastnorm(\bphi',\param)\right)\geq
\inf\left\{\pencontrastlim(\bphi'',\param)~:~\bphi''\in K,\,\|\bphi''-\hbbetadif(\param)\|\geq\eta/2\right\}  \;.
$$
Further, by definition of $B$, we have on $B^c$ that for all $\param\in\paramset$,
$\inf_{\bphi\in K}\pencontrastlim(\bphi,\param)\leq\pencontrastlim(\hbbetadif(\param),\param)$.
This and the last display show that $F\cap B^c$ is included in
$$
\left\{\inf_{\param\in\paramset}\big[
\inf\left\{\pencontrastlim(\bphi,\param)~:~\bphi\in K,\,\|\bphi-\hbbetadif(\param)\|\geq\eta/2\right\}
-\pencontrastlim(\hbbetadif(\param),\param)
\big] \leq 4\alpha \right\}\;,
$$
which, by~(\ref{eq:DefEpsilonCLT}), has probability at most $\epsilon$ for our choice of $\alpha$.
Since $K$ has been chosen so that $P(B)\leq\epsilon$, we finally get
$$
\limsup P^*(D_n)\leq 4\epsilon \; .
$$
This exactly says that $(\hbbetadif_n)$ is asymptotically  uniformly
$\tilde{\rho}$-equicontinuous in probability and the proof is achieved.
\end{proof}

\begin{proof}[Proof of Theorem~\ref{theo:ConvexConsistence}]
Let $\epsilon>0$ and denote by $B'=\{\bphi:\|\bphi-\bbeta\|\leq 2\epsilon\}$
and $B=\{\bphi:\|\bphi-\bbeta\|\leq \epsilon\}$ the balls centered at $\bbeta$ with radii $2\epsilon$ and $\epsilon$. 
We choose $\epsilon$ small enough so that $B'\subseteq V$. 
We first show that Assumption~\ref{assum:uniform-consistency} holds for $M$ defined on $\bPhi$ by 
\begin{equation}
\label{eq:MdefConvex}
M(\bphi)=\begin{cases}
\Delta(\bphi) &\text{ if } \bphi\in B\;,\\
\Delta(\bbeta)+ \alpha/2&\textrm{otherwise}\;,
\end{cases}
\end{equation}
where 
\begin{equation}\label{eq:alphaDef}
\alpha = \inf_{\bphi\in B'\setminus B} \Delta(\bphi) -\Delta(\bbeta)>0  \;.
\end{equation}
The positiveness of $\alpha$ follows from the strict convexity of $\Delta$ and Assumption~\ref{assump:convex}-(\ref{it:convex-minimum}).
Assumption~\ref{assum:uniform-consistency}-(\ref{it:SimpleConv}) follows from
Assumption~\ref{assump:convex}-(\ref{it:pointwise-conv}). 
Assumption~\ref{assum:uniform-consistency}-(\ref{it:IsolMin}) follows from the strict convexity of $\Delta$,
Assumption~\ref{assump:convex}-(\ref{it:convex-minimum}) and the definition of $M$ in~(\ref{eq:MdefConvex}).
It only remains to prove that Assumption~\ref{assum:uniform-consistency}-(\ref{it:consUnfiLowBound}) holds.  
By \cite[Theorem~10.8 ]{rockafellar:1970} and arguing as in the proof of Lemma~3 in \cite{niemiro:1992} for getting the
result in the sense of the convergence in probability, the pointwise convergence in
Assumption~\ref{assump:convex}-(\ref{it:pointwise-conv}) implies the
uniform convergence on the compact set $B'$, that is,
\begin{equation}
  \label{eq:uniformConvConvex}
  \sup_{\bphi\in B'}\left|M_n(\bphi) -\Delta(\bphi)\right|\cp 0\; .
\end{equation}
Let $\Omega'$ be a probability 1 set on which $M_n$ is convex and define
$$
A_n=\left\{\sup_{\bphi\in B'}\left|M_n(\bphi) -\Delta(\bphi)\right|\leq\alpha/4\right\}\cap \Omega' \; .
$$
The set $A_n$ is measurable since $M_n$ and $\Delta$ are convex on $\bPhi$ and thus the sup can be replaced by a sup on a 
countable dense subset of $B'$ without changing the definition of $A_n$.  
Let $\omega\in A_n$.
For all $\bphi\in B'\setminus {B}$ and $\param\in[0,L]$, we have
$M_n(\omega,\bphi) \geq \Delta(\bphi)-\alpha/4$, $\Delta(\bphi)\geq\Delta(\bbeta)+\alpha$, and, since $\bbeta\in B'$,
$\Delta(\bbeta)\geq M_n(\omega,\bbeta)-\alpha/4$. Hence
$$
\inf_{\bphi\in B'\setminus {B}}M_n(\omega,\bphi) \geq M_n(\omega,\bbeta) +\alpha/2  \;. 
$$
By convexity of the function $M_n(\omega,\cdot)$ and of the set $\bPhi$, the last display implies that 
$$
\inf_{\bphi\in \bPhi\setminus {B}}M_n(\omega,\bphi) \geq M_n(\omega,\bbeta) +\alpha/2  \;. 
$$
For all $\omega\in A_n$, using the definition of $M$ in~(\ref{eq:MdefConvex}), we thus have, for all
$\bphi\in\bPhi\setminus {B}$, 
$$
\left\{M(\bphi)-M_n(\omega,\bphi)\right\}_+= \left\{\Delta(\bbeta)+ \alpha/2-M_n(\omega,\bphi)\right\}_+
\leq \left|\Delta(\bbeta)+M_n(\omega,\bbeta)\right|\;.
$$
Using this with~(\ref{eq:uniformConvConvex}) and $P(A_n)\to1$, we get
Assumption~\ref{assum:uniform-consistency}-(\ref{it:consUnfiLowBound}). 
We conclude that Assumption~\ref{assum:uniform-consistency} holds and we obtain Assertion~(\ref{it:ConvUnifMestPenConvex}) as an
application of Theorem~\ref{theo:uniform-consistency}.

Next we show Assertion~(\ref{it:StrictlyConvex}) and thus assume that $J_n$ is strictly convex.
The proof of Assertion~(\ref{it:MStrictlyConvex}) is similar and thus omitted. 
We set
$$
L_n=\frac{\alpha}{4J_n(\bbeta)} \;,
$$
so that $L_n\to\infty$ by assumption on $J_n(\bbeta)$ and $\param J_n(\bbeta)\leq\alpha/4$ for all $\param\leq L_n$. 
Let $\omega\in A_n$. Then, for all $\bphi\in B'\setminus {B}$ and $\param\in[0,L_n]$, using that
$\pencontrast(\omega,\bphi,\param)\geq M_n(\omega,\bphi)$ and $M_n(\omega,\bbeta)=
\pencontrast(\omega,\bbeta,\param)-\param J_n(\bbeta)\geq\pencontrast(\omega,\bbeta)-\alpha/4$, we obtain  
$$
\inf_{\param\in[0,L_n]}\inf_{\bphi\in B'\setminus {B}}\pencontrast(\omega,\bphi,\param) \geq \pencontrast(\omega,\bbeta,\param) +\alpha/4  \;. 
$$
Since $J_n$ is strictly convex, so is the function $\pencontrast(\omega,\cdot,\param)$ for $\param>0$. 
By convexity of the set $\bPhi$, the previous display implies that 
for all $\param\in[0,L_n]$, the minimum of $\pencontrast(\omega,\bphi,\param)$ on $\bphi\in\bPhi$ is
attained within $B$. By strict convexity of $J_n$, this minimum is unique for $\param>0$ and we let
$\hbbeta_n(\omega,\param)$ be this unique minimum for $\param\in(0,L_n]$. For $\omega\in
A_n^c$ or $\param>L_n$, we define $\hbbeta_n(\omega,\param)=\bphi_0$, where
$\bphi_0$ is any fixed point of $\bPhi$. As for $\param=0$ and $\omega\in A_n$,  we define
$$
\hbbeta_n(\omega,0)=\liminf_{\param\downarrow0} \hbbeta_n(\param) \in B\;,
$$
where the $\liminf$ is defined component-wise in a given coordinate system of the Euclidean space containing $\bPhi$.  
Since the minimum of $\pencontrast(\omega,\bphi,\param)$ on $\bphi\in\bPhi$ is
attained within the compact set $B$, by continuity of $J_n(\bphi)$ and $M_n(\omega,\bphi)$ in $\bphi$, $\hbbeta_n(\omega,0)$ 
is a minimizer of $\pencontrast(\omega,\bphi,0)$ on $\bphi\in\bPhi$.
Thus, we have defined a r.v. $\hbbeta_n(\cdot,\param)$ for any $\param\geq0$, for which
Property~(\ref{it:UniqueMinProperty}) holds. 

To conclude the proof, we show that Property~(\ref{it:ConvexContProperty}) holds.
The continuity on $(L_n,\infty)$ for $\omega\in A_n$ and on $\Rset_+$ for $\omega\in A_n^c$  directly follows from the
definition of $\hbbeta_n(\omega,\param)$.
Let us now prove that $\hbbeta_n(\omega,\cdot)$ is continuous on $(0,L_n]$ for all $\omega\in A_n$. Since $J_n$ is convex,
it is bounded on $B$ and since $\hbbeta_n(\omega,\param)\in B$, we have
$\sup_{\param\in(0,L_n]}J_n(\hbbeta_n(\omega,\param))\leq \sup J_n(B)<\infty$. 
Let $\param$ and $\param_0$ be in $(0,L_n]$. 
We have
\begin{align*}
\pencontrast(\hbbeta_n(\omega,\param),\param_0)&\leq \pencontrast(\hbbeta_n(\omega,\param),\param)+|\param_0-\param|\;\sup J_n(B)\\
&\leq \pencontrast(\hbbeta_n(\omega,\param_0),\param)+|\param_0-\param|\;\sup J_n(B)\\
&\leq \pencontrast(\hbbeta_n(\omega,\param_0),\param_0)+2|\param_0-\param|\;\sup J_n(B)\;.
\end{align*}
Since $\pencontrast(\hbbeta_n(\omega,\param_0),\param_0)\leq\pencontrast(\hbbeta_n(\omega,\param),\param_0)$, we get that
$\pencontrast(\hbbeta_n(\omega,\param),\param_0)\to\pencontrast(\hbbeta_n(\omega,\param_0),\param_0)$ as $\param\to\param_0$.
Since, by strict convexity of $\pencontrast$, $\hbbeta_n(\omega,\param_0)$ is an isolated minimum of
$\pencontrast(\cdot,\param_0)$, this implies that $\hbbeta_n(\omega,\param)\to\hbbeta_n(\omega,\param_0)$ as
$\param\to\param_0$. The continuity of $\hbbeta_n(\omega,\cdot)$ on $(0,L_n]$ follows and the proof is achieved.  
\end{proof}

\begin{proof}[Proof of Proposition~\ref{prop:uniform-consistency}]
 Let $\epsilon>0$ and define 
$$
\alpha=\inf_{\param\in\paramset}\left[\inf_{d(\bphi,\bbeta)\geq\epsilon/2}\Lambda(\bphi;\param)-\Lambda(\bbeta(\param),\param)\right]\;.
$$
By~(\ref{it:IsolMinParam}), we have $\alpha>0$. 
Denote
$$
A_n=\left\{\sup_{\param\in\paramset}d(\hbbeta_n(\param),\bbeta(\param))\geq\epsilon\right\}\subseteq\Omega \; .
$$
For all $\omega\in A_n$, there exists $\param\in\paramset$ such that
$d(\hbbeta_n(\omega,\param),\bbeta(\param))\geq\epsilon/2$, and thus for which 
$\Lambda(\hbbeta_n(\omega,\param),\param)-\Lambda(\bbeta(\param),\param)
\geq \alpha$. Hence, for all $\omega\in A_n$, we have
$$
\sup_{\param\in\paramset}\left[\Lambda(\hbbeta_n(\omega,\param),\param)-\Lambda(\bbeta(\param),\param)\right]\geq\alpha\;.
$$
Now we write, for any $\param_0\in\paramset$,
\begin{align*}
  \Lambda(\hbbeta_n(\param_0),\param_0)-\Lambda(\bbeta(\param_0),\param_0)
&=\left\{\Lambda(\hbbeta_n(\param_0),\param_0)-\pencontrast(\hbbeta_n(\param_0),\param_0)\right\}\\
& \hspace{-2.5cm}+\left\{\pencontrast(\hbbeta_n(\param_0),\param_0)-\pencontrast(\bbeta(\param_0),\param_0)\right\}
%\\& \hspace{0.5cm}
+\left\{\pencontrast(\bbeta(\param_0),\param_0)-\Lambda(\bbeta(\param_0),\param_0)\right\}\\
&\leq \displaystyle\sup_{\bphi\in\bPhi}\sup_{\param\in\paramset}
\left\{\Lambda(\bphi;\param)-\pencontrast(\bphi,\param)\right\}_+\\
& \hspace{-2.5cm}+\displaystyle\sup_{\param\in\paramset}\displaystyle\left\{\pencontrast(\hbbeta_n(\param),\param)-\pencontrast(\bbeta(\param),\param)\right\}_+
%\\& \hspace{0.5cm}
+\displaystyle\sup_{\param\in\paramset}\left|\pencontrast(\bbeta(\param),\param)-\Lambda(\bbeta(\param),\param)\right| \;.
\end{align*}
Taking the $\sup$ in $\param_0\in\paramset$ we obtain that $A_n\subseteq A_n^{(1)}\cup A_n^{(2)}\cup A_n^{(3)}$, where 
$A_n^{(1)}=\{\sup_{\bphi\in\bPhi}\sup_{\param\in\paramset}
\left\{\Lambda(\bphi;\param)-\pencontrast(\bphi,\param)\right\}_+\geq\alpha/3\}$, and where $A_n^{(2)}$ and $A_n^{(3)}$ are defined
accordingly by using the last 2 lines of the last
display. 
Applying
$P^*(A_n)\leq P^*(A_n^{(1)})+P^*(A_n^{(2)})+P^*(A_n^{(3)})$,~(\ref{it:consUnfiLowBoundParam}),~(\ref{it:SimpleConvParam})
and~(\ref{it:DefMestimaParam}), we thus get~(\ref{eq:ConvUnifMestPenParam}), which achieves the proof. 
\end{proof}

\begin{proof}[Proof of Theorem~\ref{theo:uniform-consistency}]
We apply Proposition~\ref{prop:uniform-consistency} with $\pencontrast$ defined
by~(\ref{eq:DefPenContrast}), $\Lambda(\bphi,\param)=M(\bphi)$ and $\bbeta(\param)=\bbeta$ for all $\param$.  
Let us check the conditions in Proposition~\ref{prop:uniform-consistency}. 
Since $J_n$ is non-negative, 
$$
\left\{\Lambda(\bphi;\param)-\pencontrast(\bphi;\param)\right\}_+\leq\left\{M(\bphi)-M_n(\bphi)\right\}_+\;,
$$
and Condition~(\ref{it:consUnfiLowBoundParam}) follows from
Assumption~\ref{assum:uniform-consistency}-(\ref{it:consUnfiLowBound}). Condition~(\ref{it:SimpleConvParam}) follows from
Assumption~\ref{assum:uniform-consistency}-(\ref{it:SimpleConv}) and $J_n(\beta)\to0$. Conditions~(\ref{it:IsolMinParam})
and~(\ref{it:DefMestimaParam}) directly follow from Assumption~\ref{assum:uniform-consistency}-(\ref{it:IsolMin}) and
Eq.~(\ref{eq:DefMestimaPen}), respectively. Hence~(\ref{eq:ConvUnifMestPen}) follows from~(\ref{eq:ConvUnifMestPenParam}).
\end{proof}

\begin{proof}[Proof of Proposition~\ref{prop:vitesse}]
Denote the left-hand side
of~(\ref{eq:ConvUnifMestPenParamVitesse}) by $U_n$ and the left-hand side of~(\ref{eq:Gncond}) by $V_n$. 
Let $\delta>1$ and define  $A_n=\left\{U_n > \delta\right\}$. 
Then for all $\omega\in A_n$, we have 
\begin{equation}
  \label{eq:GncondOnAn}
\sup_{\param\in\paramset}\left|G_n(\hbbeta_n(\param),\param)\right|\leq 2n^{-1}\;\delta^{-1}\;U_n^2 \; V_n \; .  
\end{equation}
By~(\ref{it:Rncond}), using the assumed uniform $P^*$-consistency~(\ref{eq:ConvUnifMestPenParam}), there exist
non-negative random sequences $w_n$ and $W_n$ such that $w_n=o_{P^*}(1)$, $W_n=O_{P^*}(1)$ and
$$
\sqrt{n}\sup_{\param\in\paramset}\left|R_n(\hbbeta_n(\param),\param)\right|\leq (U_n\;w_n+W_n)\; ,
$$
hence, for  all $\omega\in A_n$,
$$
n\;\sup_{\param\in\paramset}\left\{d(\hbbeta_n(\param),\bbeta(\param))\left|R_n(\hbbeta_n(\param),\param)\right|\right\}
\leq U_n\;(U_n\;w_n+W_n)
\leq U_n^2\;(w_n+W_n/\delta)\; .
$$
Denote the left-hand side of (\ref{eq:DefMestimaParamVitesse}) by $S_n$. The last display,~(\ref{eq:GncondOnAn})
and~(\ref{eq:ContrasteDecompVitesse}) imply that, for all  $\omega\in A_n$ and all $\param\in\paramset$, 
$$
H(\hbbeta_n(\param),\param)\leq S_n +U_n^2\;n^{-1}\;\left\{ 2\delta^{-1}\;V_n+w_n+W_n/\delta\right\}\;.
$$
Define $B_n=\{\sup_{\param\in\paramset}d(\hbbeta_n(\param),\bbeta(\param)) >\epsilon\}$ where $\epsilon$ is the positive
number in Condition~(\ref{it:Hcond}) and denote the left-hand side of~(\ref{eq:Hcond}) by $\alpha$, which is
positive. Then, for all  $\omega\in B_n^c$,
$ \alpha\; U_n^2\leq n\,\sup_{\param\in\paramset}H(\hbbeta_n(\param),\param)$, and, using the previous display, if moreover
$\omega\in A_n$, 
$$
 \alpha\; U_n^2 \leq n\;S_n + U_n^2\;\left\{ 2\delta^{-1}\;V_n+w_n+W_n/\delta\right\} \;.
$$
Using that $P^*(B_n)\to0$, $nS_n=O_{P^*}(1)$, $V_n=O_{P^*}(1)$, $w_n=o_{P^*}(1)$ and $W_n=O_{P^*}(1)$, we easily get that
$\limsup P^*(A_n)$ can be made arbitrarily small by taking $\delta$ large
enough. Hence~(\ref{eq:ConvUnifMestPenParamVitesse}) holds. 
\end{proof}

\begin{proof}[Proof of Theorem~\ref{theo:CLTStochDiff}]
Let us define $\hbbetadif_n=\sqrt{n}(\hbbeta_n-\bbeta)$ and
\begin{equation}
  \label{eq:pencontrastnormDef}
  \pencontrastnorm(\bphi,\param)=
n\left\{\pencontrast(\bbeta(\param)+n^{-1/2}\bphi,\param)-\pencontrast(\bbeta(\param),\param) \right\}\;.
%nG_n\left(\hbbeta_n(\param)+n^{-1/2}\bphi,\param\right )+
%n H\left(\hbbeta_n(\param)+n^{-1/2}\bphi,\param\right)\;,
\end{equation}
We will apply Theorem~\ref{theo:argmaxParam} with these definitions (in the case (C-\ref{it:LinftyConvinRp})) and thus now
proceed in checking the conditions of Theorem~\ref{theo:argmaxParam} successively.
Let $K$ be a compact subset of $\bPhi$.
Using~(\ref{eq:ContrasteDecompVitesse}),~(\ref{eq:pencontrastnormLinPartDef}) and~(\ref{eq:pencontrastnormDef}), we get
$$
\pencontrastnorm(\bphi,\param)=\widehat{G}_n(\bphi,t)+nH\left(\bbeta(\param)+n^{-1/2}\bphi,\param\right)
+\sqrt{n}\|\bphi\|R_n\left(\bbeta(\param)+n^{-1/2}\bphi,\param\right)\;.
$$
Observe that by~(\ref{eq:HcondS1}) and~(\ref{eq:HcondS2}), as functions of $(\bphi,\param)$,
$$
nH\left(\bbeta(\param)+n^{-1/2}\bphi,\param\right)\to \bphi^T\Gamma(\param)\bphi\quad \text{in}\quad\ell^\infty(K\times\paramset,\Rset^p)\;.
$$
Applying~(\ref{eq:RncondS}), we obtain
$$
\sup_{(\bphi,\param)\in K\times\paramset}\sqrt{n}\|\bphi\|\,
\left|R_n\left(\bbeta(\param)+n^{-1/2}\bphi,\param\right)\right|=o_{P^*}(1).
$$
Hence using that $\widehat{G}_n\wc G$ in $\ell^\infty(K\times\paramset,\Rset^p)$, the three last displays yield
$\pencontrastnorm\wc\pencontrastlim$ in $\ell^\infty(K\times\paramset,\Rset^p)$. Since $G$ is tight in
$\ell^\infty(K\times\paramset,\Rset^p)$ by assumption,  $\pencontrastlim$ also is and thus
Condition~(\ref{it:LinftyConvContrast}) holds.
Conditions~(\ref{it:isolatedMinimumlimitCLT}) and~(\ref{it:limitTight}) hold by assumption.
Applying Proposition~\ref{prop:vitesse}, we obtain~(\ref{eq:ConvUnifMestPenParamVitesse}) and thus
Condition~(\ref{it:UnifTight}) holds.
Using~(\ref{eq:DefMestimaParamVitesseS}) with the above definitions, we get that
Condition~(\ref{it:DefArgminFalCLT}) holds.
\end{proof}

\begin{proof}[Proof of Theorem~\ref{theo:PollardWithPen}]
We shall apply Theorem \ref{theo:CLTStochDiff} for $\pencontrast$ given by~(\ref{eq:DefPenContrast}) and with
$\bbeta(\param)=\bbeta$ for all $\param\in\paramset$. 
Let us check that the assumptions of this theorem hold in this context.
Condition~(\ref{eq:DefMestimaParamVitesseS}) and  the uniform $P^*$-consistency~(\ref{eq:ConvUnifMestPenParam}) hold by
assumption. The decomposition~(\ref{eq:ContrasteDecompVitesse}) holds with
\begin{align*}
&G_n(\bphi,\param)=(\bphi-\bbeta)^TP_n\Delta+\param\left(J_n(\bphi)-J_n(\bbeta)\right)\1(\|\bphi-\bbeta\|\leq1)\;,\\
&H(\bphi,\param)=Pg(\cdot,\bphi)-Pg(\cdot,\bbeta)-(\bphi-\bbeta)^TP\Delta\;,\\
&R_n(\bphi,\param)= n^{-1/2}\nu_n\,r(\cdot,\bphi)+\param\|\bphi-\bbeta\|^{-1} \left(J_n(\bphi)-J_n(\bbeta)\right)\1(\|\bphi-\bbeta\|>1)\;.
\end{align*}

Using~(P-\ref{it:Pollardiid}) and~(P-\ref{it:PollardDelta}), we have $\sum_{k=1}^n\Delta(\xi_k)=O_P(n^{1/2})$ and,
using~(\ref{eq:condJn}), we get that Condition~(\ref{it:GncondS}) in Theorem~\ref{theo:CLTStochDiff} holds.
Observe that $H(\bphi,\param)$ does not depend on $\param$ and, by~(P-\ref{it:PollardDelta}), we have 
$$
H(\bphi,\param)=M(\bphi)-M(\bbeta)\;.
$$
Integrating $x$ with respect to $P$ in~(\ref{eq:TaylorOnG}) and using~(P-\ref{it:SDC}), we get that the first derivative of
$M$ at $\bbeta$ is zero and, by~(P-\ref{it:PollardMder}), 
$$
H(\bphi,\param)=(\bphi-\bbeta)^T\Gamma(\bphi-\bbeta)+o\left(\|\bphi-\bbeta\|^2\right)\;.
$$
Hence Condition~(\ref{it:HcondS}) in Theorem~\ref{theo:CLTStochDiff}  holds.

We have, for any sequence of positive r.v. $(r_n)$ such that $r_n\cp0$,
\begin{align*}
\sup_{\|\bphi-\bbeta\|\leq r_n}\left\{\left|n^{-1/2}\nu_n\,r(\cdot,\bphi)\right|\right\}
&\leq 
\frac{1+\sqrt{n}r_n}{\sqrt{n}}
\sup_{\|\bphi-\bbeta\|\leq r_n}
\left\{\frac{\left|\nu_n\,r(\cdot,\bphi)\right|}{1+\sqrt{n}\|\bphi-\bbeta\|}\right\} \\
&=o_P(n^{-1/2})+o_P(r_n) \; ,
\end{align*}
where the last equality follows from~(P-\ref{it:SDC}).
Observing that, for $\|\bphi-\bbeta\|\leq r_n$ and $r_n\leq 1$ the second term defining $R_n$ vanishes, we obtain
Condition~(\ref{eq:RncondS}) in Theorem~\ref{theo:CLTStochDiff}. 

Defining $\widehat{G}_n$ as in~(\ref{eq:pencontrastnormLinPartDef}) gives
$$
\widehat{G}_n(\bphi,\param)=\bphi^T\left(\sqrt{n}P_n\Delta\right)+\param \left[n\;J_{n}(\bbeta+n^{-1/2}\bphi)-n\;J_n(\bbeta)\right]\;. 
$$
Using~(P-\ref{it:Pollardiid}) and~(P-\ref{it:PollardDelta}), we have that $\sqrt{n}P_n\Delta$ converge in distribution to $W$ and,
by~(\ref{eq:penAsymp}), for any compact $K\subset\Rset^p$
$\widehat{G}_n\wc G$ in $\ell^\infty(K\times\paramset,\Rset^p)$, where
$G(\bphi,\param)=\bphi^TW+\param \;J_\infty(\bphi)$. 
This definition of $G$ and~(\ref{eq:pencontrastlimDef}) gives~(\ref{eq:pencontrastlimDefPolard}).
Hence Theorem~\ref{theo:CLTStochDiff} yields~(\ref{eq:PollardCLTStochDiff}).
\end{proof}

\begin{proof}[Proof of Lemma~\ref{lem:KFpenalties}]
We have, for all $\bphi\in\Rset^p$,
$$
\left|\sum_{k=1}^p |\phi_k|^\gamma-\sum_{k=1}^p |\beta_k|^\gamma\right|\leq C\left(\|\bphi-\bbeta\|^\gamma+\|\bphi-\bbeta\|\right)\;,
$$
where $C$ only depends on $\bbeta$ and $\gamma>0$. The bound~(\ref{eq:boundJnKF}) follows directly for $\gamma\geq1$. For $\gamma<1$, one
obtains
$$
 n\left|J_n^{(\gamma)}(\bphi)-J_n^{(\gamma)}(\bbeta)\right|\leq C'\left((\sqrt{n}\|\bphi-\bbeta\|)^\gamma
+n^{\gamma/2}\|\bphi-\bbeta\|\right)\;,
$$
and~(\ref{eq:boundJnKF}) follows by oberving that $a^\gamma\leq 1+a$ for $a\geq0$, and $n^{\gamma/2}\leq n^{1/2}$.

Relation~(\ref{eq:convJnKF}) is easily obtained by using the Taylor expansion, valid for $x\neq 0$,
$|x+y|^\gamma=|x|^\gamma+\gamma|x|^{\gamma-1}\sgn(x)\,y+O(y^2)$, which concludes the proof.
\end{proof}

%\section{Application to the regularization path of the lasso}
%\label{Applicationlasso}
%We are now in a position to prove Theorems~\ref{theo:lassoConsistence} and~\ref{theo:lassoCLT}.

\begin{proof}[Proof of Theorem~\ref{theo:lassoConsistence}]
As $\bphi\mapsto M_n(\bphi) = \frac{1}{n} \sum_{k=1}^n (y_k-\bx_k^T\bphi)^2$ 
is a convex function, we apply Theorem~\ref{theo:ConvexConsistence}. In fact, by
Assumption~\ref{assump:Consistencelasso}-(\ref{it:CondConsistencelasso1}), $M_n$ is strictly convex for $n$ large enough, and
hence the more precise Assertion~(\ref{it:MStrictlyConvex}) applies.  
We now show that Assumption~\ref{assump:convex}-(\ref{it:pointwise-conv}) holds.
\begin{equation}
M_n(\bphi)-M_n(\bbeta)=
(\bphi-\bbeta)^TC_n(\bphi-\bbeta)
-\frac{2}{n} \varepsilon_n^T\bX_n(\bphi-\bbeta)
\label{eq:DecompM_n}
\end{equation}
where $\varepsilon_n=Y_n-\bX_n\bbeta$.
Since
$$\mathbb{E}\|\bX_n^T\varepsilon_n\|^2=\mathbb{E}\left[\operatorname{Tr}(\varepsilon_n^T \bX_n
  \bX_n^T\varepsilon_n)\right]=\operatorname{Tr}\left[\bX_n \bX_n^T\right]=O(n) \; ,$$ 
by Assumption~\ref{assump:Consistencelasso}-(\ref{it:CondConsistencelasso1}),
it comes $-\frac{2}{n} \varepsilon_n^T\bX_n(\bphi-\bbeta)=O_P(n^{-1/2})$.
And furthermore, by Assumption~\ref{assump:Consistencelasso}--(\ref{it:CondConsistencelasso1}) :
\[
	M_n(\bphi)-M_n(\bbeta)\to_P (\bphi-\bbeta)^T C (\bphi-\bbeta)=\Delta(\bphi) \; .
\]
Since $C$ is positive-definite, $\Delta$ is strictly convex and Assumption~\ref{assump:convex}-(\ref{it:convex-minimum}) holds.
By definition of $\hbbeta_n(\param)$, (\ref{eq:DefMestimaPen}) holds.
Finally, the condition $J_n(\bbeta)\to 0$ holds, as the penalty is defined by $J_n(\bbeta)=\lambda_n\|\bbeta\|_1$, with
$\|\cdot\|_1$ denoting the $\ell^1$ norm.
The uniform consistency on every compact set follows as an application of Theorem~\ref{theo:ConvexConsistence}.
\end{proof}

\begin{proof}[Proof of Theorem~\ref{theo:lassoCLT}]
We apply Theorem~\ref{theo:CLTStochDiff} with $\paramset$ a compact subset of $\Rset_+$.
By definition of $\hbbeta_n(\param)$, condition (\ref{eq:DefMestimaParamVitesseS}) holds.
We just obtained uniform consistency in Theorem~\ref{theo:lassoConsistence}.
Using~(\ref{eq:DecompM_n}), we have the decomposition~(\ref{eq:ContrasteDecompVitesse}) of $\Lambda_n(\bphi,\param)$, with
\begin{gather*}
	G_n(\bphi,\param) = -2n^{-1/2}U_n^T(\bphi-\bbeta)
+\param \lambda_n \left(\|\bphi\|_1-\|\bbeta\|_1\right)\;,\\
	H(\bphi,\param)   = (\bphi-\bbeta)^TC(\bphi-\bbeta)\text{ and }
	R_n(\bphi,\param) = \|\bphi-\bbeta\|^{-1}(\bphi-\bbeta)^T(C_n-C)(\bphi-\bbeta)\;,
      \end{gather*}
where $U_n=n^{-1/2}\bX_n^T \varepsilon_n$ and $\lambda_n=n^{-1/2}$, by Assumption~\ref{assump:CLTlasso}-(\ref{it:CondCLTlasso3}).

The sequence $\{U_n\}$ converges in distribution to $U\sim\mathcal{N}(0,\sigma^2C)$
by the Lindeberg-Feller theorem and Assumption~\ref{assump:CLTlasso}. 
We have, for all $\bphi\in\Rset^p$ and $\param\in\paramset$,
$n|G_n(\bphi,\param)|\leq \sqrt{n}U_n\|\bphi-\bbeta\|+\param \sqrt{n}|\|\bphi\|_1-\|\bbeta\|_1|
\leq \|\bphi-\bbeta\|(O_P(\sqrt{n})+c\sqrt{n})$, where $c$ is a positive constant.
%, using the norm equivalence on $\Rset^p$ $\|\bphi-\bbeta\|_1 \leq c \|\bphi-\bbeta\|$. 
Hence $G_n$ satisfies~(\ref{eq:Gncond}).

Conditions (\ref{eq:HcondS1}) and (\ref{eq:HcondS2}) on $H$ are immediately verified by taking $\Gamma(\param)=C$, for all $\param\in \paramset$ and using Assumption~\ref{assump:CLTlasso}-(\ref{it:CondCLTlasso1}).

Observe that $|R_n(\bphi,\param)|\leq \rho(C_n-C)~\|\bphi-\bbeta\|$
where $\rho(C_n-C)$ is the spectral radius of $\left(C_n-C\right)$.
Since $C_n \cp C$, $\rho(C_n-C)=o_{P}(1)$ and\\
$\sup\left\{R_n(\bphi,\param), \bphi\in\bPhi,\|\bphi-\bbeta\|\leq r_n\right\}=o_P(r_n)$.
Condition (\ref{eq:RncondS}) on $R_n$ follows.

As in (\ref{eq:pencontrastnormLinPartDef}), we define
$$
  \widehat{G}_n(\bphi,\param)=nG_n\left(\bbeta+n^{-1/2}\bphi,\param\right)
  =-2U_n^T\bphi+\param n^{1/2}\sum_{j=1}^p \left\{\left|\beta_j+n^{-1/2}\phi_j\right| -|\beta_j|\right\}	 \; .
$$
For any compact $K\subseteq \Rset^p$, let $f$ map $u\in\Rset^p$ to $f[u]\in\ell^\infty(K\times \paramset)$, defined by 
$f[u](\bphi,\param)=u^T\bphi$. The map $f$ is continuous and by the continuous mapping theorem, $f(U_n)$ converges to $f(U)$
in $\ell^\infty(K\times \paramset)$. From this and~(\ref{eq:convJnKF}) with $\gamma=1$, it follows that
$\widehat{G}_n$ converges to $G$ in $\ell^\infty(K\times \paramset)$, where
$$
G(\bphi,\param)= -2U^T\bphi+\param 
	\sum_{j=1}^p \left\{\phi_j \sgn\left(\beta_j\right) {\1}_{\{\beta_j \neq 0\}} + |\phi_j|
          {\1}_{\{\beta_j = 0\}}\right\}\;.
$$
By Assumption~\ref{assump:Consistencelasso}-(\ref{it:CondConsistencelasso1}) one has 
$\pencontrastlim(\bphi,\param)\geq c_1 \|\bphi\|^2 + c_2 \|\bphi\|$ 
for all $\bphi\in\Rset^p$ and $\param\in \paramset$, with $c_1>0$ and $c_2$ a finite random variable.
Since $\pencontrastlim(0,\param)=0$, we get
$0\geq \pencontrastlim(\hbbetadif(\param),\param)\geq c_1\|\hbbetadif(\param)\|^2+c_2\|\hbbetadif(\param)\|$
thus $\hbbetadif(\param)\leq -\frac{c_2}{c_1}$. 
Condition~(\ref{it:isolatedMinimumlimitCLT}) of Theorem~\ref{theo:argmaxParam} follows immediately and so does
Condition~(\ref{it:limitTight}) of Theorem~\ref{theo:argmaxParam}, observing that $\pencontrastlim(\bphi,\param)$
is continuous in $(\bphi,\param)$ and strictly convex in $\bphi$.
The convergence (\ref{eq:lassoCLTStochDiff}) follows as an application of Theorem~\ref{theo:CLTStochDiff}.
\end{proof}

\end{document}

%% file: def.tex
  {\end{enumerate}}

  {\end{enumerate}}

\def\rmd{\mathrm{d}}

\def\cp{\stackrel{P}{\longrightarrow}}

\def\1{\mathbbm{1}}

\def\bX{\mathbf{X}}

\def\bx{\mathbf{x}}

\def\Rset{\mathbb{R}} %reels
\def\PE{\mathbb{E}} %esperance
\newcommand{\CPE}[3][]
{\ifthenelse{\equal{#1}{}}{\operatorname{E}\left[\left. #2 \, \right| #3 \right]}{\operatorname{E}_{#1}\left[\left. #2 \, \right | #3 \right]}}

\def\calS{\mathcal{S}}

\def\calB{\mathcal{B}}

\def\calF{\mathcal{F}}

\def\calD{\mathcal{D}}

\def\argmin{\mathop{\mathrm{Argmin}}}